\documentclass[12pt]{amsart}
\usepackage{amssymb, amsmath}
\usepackage{mathrsfs, amssymb, array}

\usepackage{graphicx}

\usepackage{rotating}

\def\lsimeq{\rotatebox[origin=c]{-90}{$\simeq$}}
\usepackage{hyperref}
\usepackage{verbatim}
\usepackage{mathcomp}
\usepackage{upgreek}
\usepackage[matrix,arrow,curve]{xy}
\usepackage{longtable}
\usepackage{adjustbox}
\usepackage{multirow}

\makeatletter
\@addtoreset{equation}{subsection}
\makeatother

\newcommand{\CC}{\mathbb{C}}
\newcommand{\ZZ}{\mathbb{Z}}
\newcommand{\PP}{\mathbb{P}}
\newcommand{\QQ}{\mathbb{Q}}

\newcommand{\bF}{\mathbf{F}}

\newcommand{\RR}{\mathbb{R}}
\newcommand{\Sym}{\mathfrak{S}}
\newcommand{\EEE}{\mathscr{E}}
\newcommand{\MMM}{{\mathscr{M}}}
\newcommand{\OOO}{{\mathscr{O}}}

\newcommand{\LLL}{{\mathscr{L}}}
\newcommand{\III}{{\mathscr{I}}}
\newcommand{\JJJ}{{\mathcal{J}}}
\newcommand{\Hom}{{\operatorname{Hom}}}
\newcommand{\rk}{\operatorname{rk}}
\newcommand{\Sing}{\operatorname{Sing}}
\newcommand{\Cl}{\operatorname{Cl}}

\newcommand{\Pic}{\operatorname{Pic}}

\newcommand{\Aut}{\operatorname{Aut}}
\newcommand{\Bs}{\operatorname{Bs}}
\newcommand{\Cr}{\operatorname{Cr}}
\newcommand{\Fix}{\operatorname{Fix}}

\newcommand{\mult}{\operatorname{mult}}
\newcommand{\PGL}{{\operatorname{PGL}}}
\newcommand{\PSL}{{\operatorname{PSL}}}
\newcommand{\SL}{{\operatorname{SL}}}
\newcommand{\Sp}{{\operatorname{Sp}}}
\newcommand{\PSp}{{\operatorname{PSp}}}
\newcommand{\GL}{{\operatorname{GL}}}
\newcommand{\z}{{\operatorname{z}}}

\newcommand{\g}{{\operatorname{g}}}

\newcommand{\dd}{{\operatorname{d}}}

\newcommand{\Spec}{{\operatorname{Spec}}}
\newcommand{\Am}{{\operatorname{Am}}}

\newcommand{\qsim}{\mathbin{\sim_{\scriptscriptstyle{\QQ}}}}

\newcommand{\Alt}{{\mathfrak{A}}}
\newcommand{\mumu}{{\boldsymbol{\mu}}}

\newcommand{\xref}[1]{{\rm\ref{#1}}}
\renewcommand{\emptyset}{\varnothing}

\swapnumbers
\theoremstyle{plain}
\newtheorem{theorem}[subsection]{Theorem}
\newtheorem{lemma}[subsection]{Lemma}
\newtheorem{proposition}[subsection]{Proposition}

\newtheorem{stheorem}[equation]{Theorem}
\newtheorem{corollary}[subsection]{Corollary}
\newtheorem{scorollary}[equation]{Corollary}
\newtheorem*{claim*}{Claim}
\newtheorem{sclaim}[equation]{Claim}
\newtheorem{slemma}[equation]{Lemma}
\newtheorem{sproposition}[equation]{Proposition}

\theoremstyle{definition}
\newtheorem{question}[subsection]{Question}
\newtheorem{definition}[subsection]{Definition}
\newtheorem*{definition*}{Definition}

\newtheorem{assumption}[subsection]{Assumption}
\newtheorem*{notation*}{Notation}
\newtheorem{example}[subsection]{Example}

\newtheorem{remark}[subsection]{Remark}

\newtheorem{sremark}[equation]{Remarks}

\newcounter{NN}

\newcounter{NO}

\author{J\'er\'emy Blanc}
\address{J\'er\'emy Blanc, Departement Mathematik und Informatik, Universit\"at Basel, Spiegelgasse 1, 4051 Basel, Switzerland}
\email{jeremy.blanc@unibas.ch}

\author{Ivan Cheltsov}
\address{Ivan Cheltsov, School of Mathematics, The University of Edinburgh, Edinburgh EH9 3JZ, UK
 \newline\indent
 National Research University Higher School of Economics, Moscow, Russia}
\email{I.Cheltsov@ed.ac.uk}

\author{Alexander Duncan}
\address{Alexander Duncan, Department of Mathematics, University of South Carolina, Columbia, SC 29208, USA.}
\email{duncan@math.sc.edu}

\author{Yuri Prokhorov}
\address{Yuri Prokhorov, Steklov Mathematical Institute of Russian Academy of Sciences, Moscow, Russia
 \newline\indent
 National Research University Higher School of Economics, Moscow, Russia
 \newline\indent
 Department of Algebra, Moscow State University, Russia}
\email{prokhoro@mi.ras.ru}

\title{Finite quasisimple groups acting on rationally connected threefolds}

\begin{document}

\begin{abstract}
We show that the only finite quasi-simple non-abelian groups that can faithfully act on rationally connected threefolds are
the following groups: $\Alt_5$, $\PSL_2(\bF_7)$, $\Alt_6$, $\SL_2(\bF_8)$, $\Alt_7$, $\PSp_4(\bF_3)$,
$\SL_2(\bF_{7})$, $2.\Alt_5$, $2.\Alt_6$, $3.\Alt_6$ or $6.\Alt_6$.
All of these groups with a possible exception of $2.\Alt_6$ and $6.\Alt_6$ indeed act on some rationally connected threefolds.
\end{abstract}
\maketitle
\tableofcontents
\section{Introduction}

The complex projective plane $\PP^2$ and projective space $\PP^3$ are among the most basic objects of geometry.
They provide motivation for the study of two exceptionally complicated objects,
the groups~$\Cr_2(\CC)$ and $\Cr_3(\CC)$ of their birational transformations,
known as the \emph{plane Cremona group} and the \emph{space Cremona group}, respectively.
The group $\Cr_2(\CC)$ has been studied intensively over the last two centuries, and many facts about it were established.
The structure of the group $\Cr_3(\CC)$ is much more complicated and mysterious.
It still resists most attempts to study its global structure.

One approach to studying Cremona groups is by means of their finite subgroups.
An almost complete classification of finite subgroups of the plane Cremona group $\Cr_2(\CC)$ was obtained Dolgachev and Iskovskikh in \cite{DolgachevIskovskikh} (see also \cite{Blanc2011,Tsygankov2011e,Tsygankov2013} for further developments).
For example, this classification implies that there are exactly three isomorphism classes of non-abelian simple finite subgroups of $\Cr_2(\CC)$, namely those of $\Alt_5$, $\PSL_2(\bF_7)$ and $\Alt_6$.

Recent achievements in three-dimensional birational geometry allowed Prokhorov to prove

\begin{theorem}[{\cite[Theorem~1.1]{Prokhorov2012}}]
\label{theorem:Prokhorov-1}
There are exactly six isomorphism classes of non-abelian simple finite subgroups of $\Cr_3(\CC)$, given by
\[\Alt_5,\ \PSL_2(\bF_7),\ \Alt_6,\ \SL_2(\bF_8),\ \Alt_7,\ \text{and }\ \PSp_4(\bF_3).\]
\end{theorem}

This classification became possible thanks to the general observation that a birational action of a finite group $G$ on projective space
can be regularized, that is, replaced by a regular action of this group on some more complicated rational threefold.
Thus, instead of studying finite subgroups in the space Cremona group,
one can consider a more general (and perhaps more natural) problem:
\begin{question}
What are the isomorphism classes of finite groups acting faithfully on rationally connected threefolds?
\end{question}
The classical technique to study this (hard) problem goes as follows.
Let $X$ be a rationally connected threefold faithfully acted on by a finite group~$G$.
Taking the $G$-equivariant resolution of singularities and
applying the $G$-equivariant Minimal Model Program, we can replace $X$ by a $G$-Mori fibre space.
Thus, we may assume that $X$ has terminal singularities,
every $G$-invariant Weil divisor on $X$ is $\QQ$-Cartier,
and there exists a $G$-equivariant surjective morphism
\begin{equation*}
\phi\colon X\to Z
\end{equation*}
whose general fibres are Fano varieties,
and the morphism $\phi$ is minimal in the following sense: $\rk\Pic(X/Z)^G=1$.
If $Z$ is a point, then $X$ is a Fano threefold, so that we say that $X$ is a \emph{$G\QQ$-Fano threefold}.
Similarly, if $Z=\PP^1$, then $X$ is fibred in del Pezzo surfaces, and we say that $\phi$ is a \emph{$G$-del Pezzo fibration}.
Finally, if $Z$ is a rational surface, then the general geometric fibre of $\phi$ is $\PP^1$, and $\phi$ is said to be a $G$-conic bundle.
In this case, we may assume that both $X$ and $Z$ are smooth due to a recent result of Avilov \cite{Avilov}. A priori, the threefold $X$ can be non-rational.
However, if $X$ is rational, then any birational map $X\dasharrow\PP^3$ induces an embedding
\begin{equation*}
G\hookrightarrow\Cr_3\big(\CC\big).
\end{equation*}
Vice versa, every finite subgroup of $\Cr_3(\CC)$ arises in this way.
Thus, keeping in mind that every smooth cubic threefold is non-rational,
we see that Theorem~\ref{theorem:Prokhorov-1} follows from the following (more explicit) result:

\begin{theorem}[{\cite[Theorem~1.5]{Prokhorov2012}}]
\label{theorem:Prokhorov-2}
Let $X$ be a Fano threefold with terminal singularities, and let $G$ be a finite non-abelian simple subgroup in $\Aut(X)$
such that \mbox{$\rk\Cl(X)^G=1$}.
Suppose also that $G$ is not isomorphic to $\Alt_5$, $\PSL_2(\bF_7)$ or $\Alt_6$.
Then the following possibilities hold:
\begin{enumerate}
\item $G\simeq\Alt_7$, and $X$ is the unique smooth intersection of a quadric and a cubic in $\PP^5$ that admits a
faithful action of the group $\Alt_7$;

\item $G\simeq\Alt_7$, and $X$ is $\PP^3$;

\item $G\simeq\PSp_4(\bF_3)$, and $X$ is $\PP^3$;

\item $G\simeq\PSp_4(\bF_3)$, and $X$ is the Burkhardt quartic in $\PP^4$;

\item $G\simeq\SL_2(\bF_8)$, and $X$ is the unique smooth Fano threefold of Picard rank $1$ and genus $7$
that admits a faithful action of the group $\SL_2(\bF_8)$;

\item $G\simeq\PSL_2(\bF_{11})$, and $X$ is the Klein cubic threefold in $\PP^4$;

\item $G\simeq\PSL_2(\bF_{11})$, and $X$ is the unique smooth Fano threefold of Picard rank $1$ and genus $8$
that admits a faithful action of the group $\PSL_2(\bF_{11})$ $($which
is non-equivariantly birational to the Klein cubic threefold$)$.
\end{enumerate}
\end{theorem}

In this text, we extend the study of simple groups to quasi-simple groups.

\begin{definition}
\label{definition:quasi-simple}
A group is said to be \emph{quasi-simple} if it is perfect, that is, it equals its commutator subgroup,
and the quotient of the group by its center is a simple non-abelian group.
\end{definition}
Taking the quotient by the center, which is again a rationally connected threefold, it follows from Theorem~\ref{theorem:Prokhorov-2} that
the only finite quasi-simple non-simple group that can (possibly) faithfully act on rationally connected threefolds are $2.\Alt_5$ or
\begin{equation}
\label{eq:list}
\SL_2(\bF_7),\hspace{3pt} \SL_2(\bF_{11}),\hspace{3pt} \Sp_4(\bF_{3}),\hspace{3pt}
n.\Alt_6,\hspace{3pt} n.\Alt_7\ \text{with $n=2,3,6$.}
\end{equation}
As the group $2.\Alt_5$ is a subgroup of $\SL_2(\CC)$, there are many ways to embed it into $\Cr_2(\CC)$ (see \cite{Tsygankov2013}), and hence in $\Cr_3(\CC)$. However, none of the groups of \eqref{eq:list} embedds in $\Cr_2(\CC)$ (see Theorem~\ref{th:Cr2simple}). Some of them indeed act on rationally connected threefolds. The goal of this paper is to prove the following result:
\begin{theorem}
\label{theorem:main}
Every finite quasi-simple non-simple group that faithfully acts on a rationally connected threefold is isomorphic to one of the following groups
\[\SL_2(\bF_{7}),\ 2.\Alt_5,\ 2.\Alt_6,\ 3.\Alt_6,\ \text{and }\ 6.\Alt_6.\]
Moreover, the groups $2.\Alt_5$ and $3.\Alt_6$ act faithfully on rational threefolds, and the group $\SL_2(\bF_{7})$ acts faithfully on rationally connected threefolds.
\end{theorem}

Unfortunately, we do not know whether the groups $2.\Alt_6$ and $6.\Alt_6$ can act on a rationally connected threefold or not (see Section~\ref{section:conic-bundles} for a discussion), and do not know if $\SL_2(\bF_{7})$ can act on a rational threefold.

We finish this introduction by giving examples that prove the existence part of Theorem~\ref{theorem:main} (we omit the case of $2.\Alt_5$, already explained above).

\begin{example}\label{ex:3A6}
Let $G=3.\Alt_6$ act on $V:=\CC^3$ and let $\phi(x_1,x_2,x_3)$
be the invariant of degree $6$ (unique up to scalar multiplication). Then we have the following induced actions:
\begin{enumerate}
\item
on $\PP^3=\PP(V\oplus \CC)$,
\item
on the hypersurface $X_6\subset \PP(1^3,2,2)$ given by
$\phi+ y_1^3+y_2^3=0$,
\item \label{ex:3A6-V1}
on the hypersurface $X_6\subset \PP(1^3,2,3)$ given by
$z^2+y^3+\phi=0$,
\item \label{ex:3A6-sextic-dc}
on the hypersurface $X_6\subset \PP(1^4,3)$ given by
$\phi+ x_4^6=y^2$,
\item
on $\mathbb{P}(\mathcal{O}_{\mathbb{P}^2}(d) \oplus \mathcal{O}_{\mathbb{P}^2})$ where $d\ge 1$ is not a multiple of $3$.
\end{enumerate}
\end{example}

\begin{example}\label{ex:sl27}
The group $\SL_2(\bF_{7})$ has an irreducible four-dimensional representation, which makes it then acts faithfully on $\PP^4$ and on $\PP(1,1,1,1,3)$.
\begin{enumerate}
\item
\label{ex:sl27:double6}
The weighted projective space $\PP(1,1,1,1,3)$ contains a $\SL_2(\bF_{7})$-invariant sextic hypersurface (see \cite{Edge47,MaSl73}).
This hypersurface, that we denote by $X$, is unique. In appropriate quasihomogeneous coordinates, the threefold $X$ is given by
\[\begin{array}{rcl}
w^2&=&8x^6-20x^3yzt-10x^2y^3z-10x^2yt^3-10x^2z^3t-10xy^3t^2\\
&&-10xy^2z^3-10xz^2t^3-y^5t-15y^2z^2t^2-yz^5-zt^5,
\end{array}\]
where $x$, $y$, $z$, $t$ are coordinates of weight $1$, and $w$ is a coordinate of weight $3$.
One can check that $X$ is smooth, so that $X$ is a smooth Fano threefold with $\Pic(X)=\ZZ\cdot K_X$ and $-K_X^3=2$.
Note that $X$ is non-rational (see \cite{Iskovskikh}).
\item\label{ex:sl27-quartic}
The group $\SL_2(\bF_7)$ also acts on the smooth quartic $X_4\subset \PP^4$ given by
$y^4=\phi_4$, where $\phi_4(x_1,x_2,x_3, x_4)$ is an invariant of degree $4$.
This variety is also non-rational \cite{IskovskikhManin}.
\end{enumerate}
\end{example}

\subsection*{Acknowledgements}
This research was supported through the programme ``Research in Pairs'' by the Mathematisches Forschungsinstitut Oberwolfach in 2018.
We would like to thank the institute for wonderful working conditions. J\'er\'emy Blanc was partially supported by the Swiss National Science Foundation Grant  ``Birational transformations of threefolds'' 200020\_178807.
Ivan Cheltsov was partially supported by the Royal Society grant No. 	IES\textbackslash R1\textbackslash 180205, and the Russian Academic Excellence Project 5-100.
Alexander Duncan was partially supported by National Security Agency grant H98230-16-1-0309.
Yuri Prokhorov was partially supported by the Russian Academic Excellence Project 5-100.
The authors would like to thank Josef Schicho for his help with Remark~\ref{remark:KE-cubic}.

\section{Preliminaries}
\label{section:preliminaries}

\subsection{Notation}
Throughout this paper the ground field is supposed to be the field of complex numbers $\CC$.
We employ the following standard notations used in the group theory:

\begin{itemize}
\item
$\mumu_n$ denotes the multiplicative group of order $n$ (in $\CC^*$),
\item
$\Alt_n$ denotes the alternating group of degree $n$,
\item
$\SL_n (\bF_q)$ (resp.~$\PSL_n (\bF_q)$) denotes the special linear group (resp.~projective special linear group) over
the finite field $\bF_q$,
\item
$\Sp_n (\bF_q)$ (resp.~$\PSp_n (\bF_q)$) denotes the symplectic group (resp.~projective symplectic group) over
the finite field $\bF_q$,
\item
$n.G$ denotes a non-split central extension of $G$ by $\mumu_n$,
\item
$\z(G)$ (resp.~$[G, G]$) denotes the center (resp.~the commutator subgroup) of a group $G$.
\end{itemize}
All simple groups are supposed to be non-cyclic.

\begin{lemma} \label{lem:Hurwitz}
Let $C$ be a smooth curve with a faithful action of a finite group $G$
of genus $g<\frac{1}{4}\lvert G\rvert$.
Then
\[
\frac{2g-2}{\lvert G\rvert} + 2 = \sum_{r} c_r\left(1-\frac{1}{r}\right)
\]
where $r$ varies over the orders of cyclic subgroups of $G$,
and $\{ c_r\}$ are non-negative integers.
\end{lemma}

\begin{proof}
This is a standard consequence of the Riemann-Hurwitz formula
for the quotient morphism $C \to C/G$:
\[
2g-2=\lvert G\rvert(2g_q-2)+\sum_p (e_p-1)
\]
where $p$ varies over the branch points,
$e_p$ are the ramification indices,
and $g_q$ is the genus of the quotient.
Recall that the stabilizers of all points must be cyclic,
so we get a contribution of the form $\lvert G\rvert(r-1)/r=e_p-1$ for each
$G$-orbit (free orbits contributing $0$).
Solving for $g$, we see that $g_q = 0$ or else $g \ge \frac{1}{4}\lvert G\rvert$.
\end{proof}

\begin{lemma}[see, e.\,g., {\cite[p.~98]{Cartan}}]
\label{lemma:stabilizer-faithful}
Let $X$ be an irreducible algebraic variety, let $P$ be a point in $X$,
and let $G$ be a finite group in $\Aut(X)$ that fixes the point $P$.
Then the natural linear action of $G$
on the Zariski tangent space $T_{P,X}$ is faithful.
\end{lemma}

\begin{theorem}[{\cite{Blichfeldt}}]
\label{th:GL3}
Let $G\subset\GL_3(\CC)$ be a finite quasi-simple subgroup.
Then $G$ is isomorphic to one of the following groups:
\begin{equation*}
2.\Alt_5,\quad \Alt_5,\quad 3.\Alt_6, \quad \PSL_2 (\bF_7).
\end{equation*}
\end{theorem}

\begin{theorem}[{\cite{DolgachevIskovskikh}}]
\label{th:Cr2simple}
Let $G\subset \Cr_2(\CC)$ be a finite quasi-simple subgroup such that $G/\z(G)\not\simeq \Alt_5$.
Then $G$ is conjugate to one of the following actions:
\begin{enumerate}
\item
$\Alt_6$ acting on $\PP^2$,
\item
$\PSL_2(\bF_7)$ acting on $\PP^2$,
\item
$\PSL_2(\bF_7)$ acting on a unique del Pezzo surface of degree $2$.
\end{enumerate}
In particular, $G$ is simple.
\end{theorem}

\begin{lemma}\label{Lemma-P3-Q}
Let $G$ be a group isomorphic to one in the list \eqref{eq:list}.
\begin{enumerate}
\item \label{Lemma-P3-Q:1}
If $G\subset \Aut(\PP^3)$, then $G\simeq 3.\Alt_6$ and the action is induced
by the reducible representation $V=V_1\oplus V_3$ with $\dim V_1=1$, $\dim V_3=3$.
\item \label{Lemma-P3-Q:hyp}
If $G\subset \Aut(X)$, where $X=X_d\subset \PP^4$ is a an irreducible hypersurface of degree
$d\le 4$, then
$G\simeq \SL_2(7)$, $X$ is smooth quartic, and the action is induced
by the reducible representation $V=V_1\oplus V_4$ with $\dim V_1=1$, $\dim V_4=4$.
\item \label{Lemma-P3-Q:Q}
If $G\subset \Aut(\PP^5)$, then there exists no $G$-invariant
 quadric $Q\subset \PP^5$ of corank $\le 2$.
 \item \label{Lemma-P3-Q:Q1}
If $G\subset \Aut(\PP^5)$, then there exists no $G$-invariant
 irreducible complete intersection of two quadrics.
\end{enumerate}
\end{lemma}

\begin{proof}
The assertion \ref{Lemma-P3-Q:1} follows from Table~\ref{table}.

\ref{Lemma-P3-Q:hyp}
Regard $\PP^4$ as the projectivization of a vector space $V=\CC^5$ and consider a lifting $\tilde G\subset \SL(V)$,
where $\tilde G$ is quasi-simple.
Since $G$ is not simple, $z(\tilde G)$ is not a subgroup of scalar matrices, i.e.~there exists a non-trivial decomposition $V= V'\oplus V''$ of $\tilde G$-modules,
where $\dim V'>\dim V''$. Then $\dim V''\le 2$ and so the action of $\tilde G$ on $V''$ must be trivial and
on $V'$ it is faithful with $\dim V'=3$ or $4$.
Then $\tilde G$ has an invariant of degree $\le d$ on $V'$.
From Table~\ref{table}, the only possibility is $\tilde G\simeq \SL_2(7)$
and $d=4$.

\ref{Lemma-P3-Q:Q} and \ref{Lemma-P3-Q:Q1} follow from Table~\ref{table}.
\end{proof}

\begin{lemma}
\label{lemma:fixed-point}
Let $X$ be a threefold with terminal singularities and a faithful action
of a group $G$ from the list \eqref{eq:list}. Assume that $X$ has a $G$-fixed point $P$.
Then one of the following holds:
\begin{enumerate}
\item
$P\in X$ is smooth and
$G\simeq 3.\Alt_6$,
\item
$P\in X$ is of type $\frac12(1,1,1)$ and
$G\simeq 3.\Alt_6$.
\end{enumerate}
\end{lemma}
\begin{proof}
First, consider the case where $P\in X$ is Gorenstein. The group $G$ faithfully acts on the tangent space
$T_{P,X}$. If $P\in X$ is smooth, then $\dim T_{P,X}=3$ and $G\simeq 3.\Alt_6$
by Theorem~\ref{th:GL3}. If $P\in X$ is singular, then
there exists an analytic equivariant embedding $(X,P) \subset (T,0)$, where $T\simeq \CC^4$
and the action on $T$ is linear.
Let $\phi(x_1,\dots,x_4)=0$ be the (invariant) equation of $X$ in $T$.
Write $\phi=\sum \phi_d$, where $\phi_d$ is homogeneous of degree $d$.
By the classification of terminal singularities~\cite{Reid-YPG1987},
we conclude $\phi_2\neq 0$.
If moreover $G \hookrightarrow \GL_4(\CC)=\GL(T)$ is irreducible, then
the group $G/\z(G)$ faithfully acts on $\PP(T)=\PP^3$.
In this case $\phi_2=0$ defines an invariant quadric $Q\subset \PP^3$
which must be smooth. Thus $Q\simeq \PP^1\times \PP^1$ and then the simple group $G/\z(G)$ embedds into $\PGL_2$; impossible for $G$ in the list \eqref{eq:list}.
Let $G \hookrightarrow \GL_4(\CC)=\GL(T)$ be reducible.
We have a decomposition $T=T'\oplus T''$, where $T'$ is irreducible faithful with $\dim T'<4$.
If $\dim T'=3$, then $G\simeq 3.\Alt_6$. Again $3.\Alt_6$
has no invariants of degree $\le 3$ on $T'$, so $\phi_2=x_4^2$ and $\phi_3=\lambda x_4^3$.
This contradicts the classification of terminal singularities.
Hence $\dim T'=2$ and $G\simeq 2.\Alt_5$.
Again we have a contradiction.

Consider the case where $(X,P)$ is a singularity of index $r>1$.
Let $\pi: (X^\sharp, P^\sharp)\to (X,P)$ be the index one cover and let
$G^\sharp\subset \Aut(X^\sharp, P^\sharp)$ be the natural lifting of $G$.
We have an exact sequence
\begin{equation*}
1 \longrightarrow\mumu_r \longrightarrow G^\sharp \overset{\nu}\longrightarrow G \longrightarrow 1.
\end{equation*}
Since $G$ is a quasi-simple group and $\Aut(\mumu_r)$ is abelian, this is a central extension.
Let $Z^\sharp\subset G^\sharp$ be the
preimage of $\z(G)$. Since $\z(G)$ is cyclic, $Z^\sharp$ is an abelian group with two generators
and one of these generators is of order 2 or 3. In this situation, the automorphism group
$\Aut(Z^\sharp)$ is solvable.
Hence $Z^\sharp$ coincides with the center of $G^\sharp$.
Thus either $G^\sharp=\mumu_r\times G$ or $Z^\sharp\simeq \mumu_6$ and $G/\z(G)\simeq \Alt_6$ or $\Alt_7$.
In both cases there exists a quasi-simple subgroup $G'\subset G^\sharp$ such that $\nu(G')=G$.
By the above $G'\simeq 3.\Alt_6$ and $P^\sharp\in X^\sharp$ is smooth.
Since the representation of $G'$ on $T_{P^\sharp, X^\sharp}$ is irreducible,
$r=2$.
\end{proof}

\subsection{Varieties of minimal degree}
\begin{stheorem}[F.~Enriques, see e.g.~{\cite{EisenbudHarris1987}}]
\label{th:min-degree}
Let $Y=Y_d\subset \PP^N$ be an irreducible subvariety of degree $d$ and dimension $n$
which is not contained in a hyperplane. Then
$d\ge N-n+1$ and the equality holds if and only if $Y$ is one of the following:
\begin{enumerate}
\item \label{th:Enriques-P3}
$Y=\PP^N$;
\item \label{th:Enriques-Q}
$Y=Y_2\subset \PP^N$, a smooth quadric;
\item \label{th:Enriques-scroll}
a rational scroll $\PP_{\PP^1} (\EEE)$, where $\EEE$ is an ample rank $n$ vector bundle on $\PP^1$,
embedded by the linear system $|\OOO(1)|$;
\item \label{th:Enriques-Veronese}
a Veronese surface $F_4\subset \PP^5$;
\item \label{th:Enriques-cone0}
a cone over one of the varieties from \xref{th:Enriques-Q},
\xref{th:Enriques-scroll} or \xref{th:Enriques-Veronese}.
\end{enumerate}
\end{stheorem}

\subsection{Group actions on K3 surfaces}

\begin{stheorem}[{\cite{Mukai1988}}]
\label{th:actions-K3}
Let a finite quasi-simple group $G$ faithfully act on a K3 surface $S$.
Then $G\simeq \Alt_5$, $\Alt_6$ or $\PSL_2(\bF_7)$. In particular, $G$ is simple.
Moreover, if $G\simeq \Alt_6$ or $\PSL_2(\bF_7)$, then $\rk\Pic(S)=20$ and $\rk\Pic(S)^G=1$.
\end{stheorem}

\begin{scorollary}
\label{cor:actions-K3}
Let $S$ be a K3 surface with at worst Du Val singularities.
Suppose $S$ admits a faithful action of a quasi-simple group $G$,
where $G\not \simeq \Alt_5$. Then $S$ is smooth.
\end{scorollary}

\subsection{Linearizations}

Let $X$ be a proper complex variety with a faithful action of a finite
group $G$. Let $\EEE$ be a vector bundle on $X$.
We say that $\EEE$ is $G$-invariant if there exist isomorphisms
$\phi_g : g^\ast \EEE \to \EEE$ for every $g \in G$.
We say that $\EEE$ is \emph{$G$-linearizable} if $\EEE$ is $G$-invariant and one
can select the isomorphisms $\{ \phi_g \}_{g \in G}$
such that $\phi_{gh} = \phi_h \circ h^\ast( \phi_g)$ for all $g,h \in G$.
Equivalently, this means that $G$ acts on the total space of $\EEE$
linearly on the fibers and the projection to $X$ is equivariant.
The particular choice of action on $\EEE$ is called a
\emph{linearization}.

\begin{proposition} \label{prop:canLinearizable}
If $X$ is a smooth $G$-variety, then the canonical line bundle has a
canonical linearization.
\end{proposition}

\begin{proof}
The points of the total space of the tangent bundle $TX$ are of the
form $(x,t)$ where $x \in X$ and $t \in T_xX$.
For $g \in G$, $g(x,t):=(g(x),dg(t))$ defines an action on $TX$ which is
linear on the fibers. Thus $TX$ is linearizable, and so is the
canonical bundle.
\end{proof}

\begin{lemma}[{\cite[Lemma~5.11]{DolgachevIskovskikh}}]
\label{lemma:DI}
Let $f: X\to Y$ be a double cover of smooth varieties whose branch
divisor $B\subset Y$ is given by an invertible sheaf $\LLL$ together with a section
$s_B\in H^0(Y, \LLL^{\otimes 2})$ $($see \cite{Wavrik1968}$)$ and let $\tau$ is the Galois involution.
Suppose a group $G$ acts on $Y$ leaving invariant $B$.
Then there exists a subgroup $G'\subset \Aut(X)$
fitting into the exact sequence
\begin{equation*}
1 \longrightarrow \langle\tau\rangle \longrightarrow G' \overset{\delta}{\longrightarrow} G \longrightarrow 1,
\end{equation*}
where $\delta$ is induced by $G'\to \Aut(Y)$. The sequence splits
if and only if
$\LLL$ admits a $G$-linearization and in the corresponding representation of $G$ in $H^0(Y, \LLL^{\otimes 2})$ the section $s_B$ is $G$-invariant.
\end{lemma}

\section{Gorenstein Fano threefolds}

\subsection{Special Fano threefolds}

\begin{sproposition}
\label{Proposition-rho>1}
Let $G$ be a group from the list \eqref{eq:list} and let $X$ be a $G$-Fano threefold
\textup(with only terminal Gorenstein singularities\textup).
Then $\rk\Pic(X)=1$.
\end{sproposition}
\begin{proof}
The lattice $\Pic(X)$ is equipped with non-degenerate $G$-invariant pairing $(D_1,D_2)=-K_X\cdot D_1\cdot D_1$.
According to \cite{Prokhorov-GFano-2} we have $\rk\Pic(X)\le 4$.
The group $G$ acts on the orthogonal complement $W=K_X^{\perp}$ of $K_X$ in $\Pic(X)$.
Since $X$ is a $G$-Fano, $W^G=0$. But groups as in \eqref{eq:list} have no rational
representations of dimension $\le 3$, a contradiction.
\end{proof}

Fano threefolds with Fano index $\iota(X)=2$ are also called \emph{del Pezzo threefolds}.
\begin{sproposition}
\label{Proposition-iota=2}
Let $X$ be a Fano threefold with
with at worst canonical Gorenstein singularities.
Assume that $\iota(X)=2$ and $(-\frac12 K_X)^3\le 4$.
Furthermore, assume that $\Aut(X)$ contains a subgroup $G$ from the list \eqref{eq:list}.
Then $(X,G)$ is as in Example~\xref{ex:3A6}\xref{ex:3A6-V1}.
\end{sproposition}
\begin{proof}
Let $A=-\frac12 K_X$ and $\dd(X):=A^3$.
Consider the possibilities for $\dd(X)$ case by case.
We use the classification of del Pezzo threefolds \cite{Fujita1990}, \cite{Shin1989}.

\par\medskip\noindent
{\bf Case $\dd(X)=1$.}
Here $\Bs|A|=\{P\}$ and $G$ faithfully acts on $T_{P,X}$.
Hence $G\simeq 3.\Alt_6$ by Theorem~\ref{th:GL3}.
Let $\tilde G$ be the universal central extension of $G$
(see, for example, \S{33}~of~\cite{Aschbacher}).
Then the action of $G$ on $X$ lifts to an
action of $\tilde G$ on $H^0(X,tA)$ for any $t$.
Hence $\tilde G$ acts on the graded algebra
\begin{equation*}
R(X,A):=\bigoplus_{t\ge 0} H^0(X,tA).
\end{equation*}
In our case $R(X,A)$ is generated by its elements $x_1,x_2,x_2,y,z$
with $\deg x_i=1$, $\deg y=2$, $\deg z=3$
and a unique relation of degree $6$.

There exists a natural isomorphism $H^0(X,A)\simeq T_{P,X}$.
The subspace
\[
S^2 (H^0(X, A))\subset H^0(X,-K_X)
\]
is invariant.
Since $\dim H^0(X,-K_X)=7$ and $\dim S^2 (H^0(X, A))=6$,
the element $y\in H^0(X,-K_X)$ can be taken to be a relative invariant
of $\tilde G$.
Similarly, the subspace of the $10$-dimensional space $H^0(X,3A)$
generated by $x_1,x_2,x_3,y$ is invariant and of codimension 1.
Hence the element $z\in H^0(X,3A)$ can be taken to be a relative
invariant of $\tilde G$.
Since the action on $x_1,x_2,x_3$ has no invariants of degree $<6$
and has a unique invariant $\phi_6$ of degree $6$,
$X\subset \PP(1^3,2,3)$ is given by the equation
\begin{equation*}
z^2+ y^3 +\phi_6
\end{equation*}
and so we are in the situation of Example~\ref{ex:3A6}\ref{ex:3A6-V1}.

\par\medskip\noindent
{\bf Case $\dd(X)=2$.}
In this case the map given by the linear system $|A|$
is a finite morphism $\Phi_{|A|}: X\to \PP^3$
of degree $2$ branched over a quartic $B\subset \PP^3$.
The group $G$ acts non-trivially on $\PP^3$. Therefore, $G/\z(G)$
is either $\PSL_2(7)$ or $\Alt_6$.
In the case $G/\z(G)=\Alt_6$, the group $G$ has no invariant quartic.
Hence $G/\z(G)=\PSL_2(7)$ and $G\simeq \SL_2(7)$.

Similar to the above considered case
$R(X,A)$ is generated by $x_1,\dots,x_4,y$ with $\deg x_i=1$, $\deg y=2$ with
a unique relation of degree $4$.
We may take $y$ to be a relative invariant for $\tilde G$.
Hence $X\subset \PP(1^4,2)$ can be given by the equation
\begin{equation*}
y^2 +y\phi_2+\phi_4=0
\end{equation*}
Since the action on $x_1,\dots,x_4$ has no invariants of degree $\le 2$,
$\phi_2=0$.
Now we see that $\z(G)$ acts trivially on $X$ (cf. Lemma~\ref{lemma:DI}), a contradiction.

\par\medskip\noindent{\bf Case $\dd(X)=3$.}
This case does not occur by Lemma~\ref{Lemma-P3-Q}\ref{Lemma-P3-Q:hyp}.

\par\medskip\noindent{\bf Case $\dd(X)=4$.}
In this case $X$ is an intersection of two quadrics $Q_1$ and $Q_2$ in
$\PP^5$,
which is impossible by Lemma~\ref{Lemma-P3-Q}\ref{Lemma-P3-Q:Q1}.
\end{proof}

\begin{scorollary}
\label{corollary:iota=2G}
Let $G$ be a group from the list \eqref{eq:list} and let $X$ be a $G$-Fano threefold with $\iota(X)=2$.
Then $X$ is as in Example~\xref{ex:3A6}\xref{ex:3A6-V1}.
\end{scorollary}

\begin{proof}
By Proposition~\ref{Proposition-iota=2} we may assume that $(-\frac12K_X)^3\ge 5$
and by Proposition~\ref{Proposition-rho>1} we have $\rk \Pic(X)=1$.
Hence by \cite[Theorem~1.7]{Prokhorov-GFano-1} and \cite[Proposition~7.1.10]{CheltsovShramov2015CRC} we have
 $(-\frac12K_X)^3= 5$, $X$ is smooth, and $\Aut(X)\simeq \PSL_2(\CC)$. On the other hand,
$\PSL_2(\CC)$ does not contain finite non-solvable groups different from $\Alt_5$, a contradiction.
\end{proof}

\begin{lemma}[cf. {\cite[Lemma~5.2]{Prokhorov2012}}]
\label{lemma:base-points}
Let $X$ be a Fano threefold with at worst
canonical Gorenstein singularities.
Assume that $\Aut(X)$ contains a subgroup $G$ as in the list \eqref{eq:list}.
Then the linear system $|-K_X |$ is base point free.
\end{lemma}
\begin{proof}
Assume that $\dim \Bs|-K_X|=0$, then by \cite{Shin1989} $\Bs|-K_X|$
is a single point, say $P$, and $X$ has at $P$ terminal Gorenstein singularity of type $cA_1$.
By Lemma~\ref{lemma:fixed-point} this is impossible.

Thus $\dim \Bs|-K_X|>0$, then by \cite{Shin1989} $\Bs|-K_X|$ is a smooth rational curve $C$
contained in the smooth locus of $X$.
The action of $G$ on $C$ must be trivial and we obtain a contradiction
as above.
\end{proof}

\begin{lemma}\label{lemma:very-ample}
Let $X$ be a Fano threefold with at worst
canonical Gorenstein singularities.
Assume that $\Aut(X)$ contains a subgroup $G$ from the list \eqref{eq:list}.
If the linear system $|-K_X |$ is not very ample,
then $(X,G)$ is as in Examples~\xref{ex:sl27}\xref{ex:sl27:double6}, \xref{ex:3A6}\xref{ex:3A6-V1} or
\xref{ex:3A6}\xref{ex:3A6-sextic-dc}.
\end{lemma}
\begin{proof}[Proof $($cf. {\cite[Lemma~5.3]{Prokhorov2012}}$)$]
Assume that the linear system \mbox{$|-K_X|$} defines a morphism $\varphi : X
\to \PP^{g+1}$ which is not an embedding. Let $Y = \varphi(X)$ and let $\bar G$ be the
image of $G$ in $\Aut(X)$. Then $\varphi$ is a
double cover and $Y \subset \PP^{g+1}$ is a subvariety of degree $g - 1$ (see \cite{Iskovskikh-1980-Anticanonical}
and \cite{Przhiyalkovskij-Cheltsov-Shramov-2005en}).
Hence either $\bar G\simeq G$ or $\bar G$ is the quotient of $G$ by a subgroup of order $2$.
Let $H$ be the class of a hyperplane section of $Y$ and let $B\subset Y$ be the branch divisor. Then
\begin{equation}
\label{eq:hyp-B}
-K_X=\varphi^* H,\quad K_X=\varphi^*\left(K_Y+\textstyle \frac12 B\right),\quad K_Y+H+\textstyle \frac12 B=0.
\end{equation}
Apply Theorem~\ref{th:min-degree}.
The case where $Y$ is a quadric (case~\ref{th:min-degree}\ref{th:Enriques-Q})
does not occur by Lemma~\ref{Lemma-P3-Q}\ref{Lemma-P3-Q:hyp} and \cite[Lemma~3.6]{Prokhorov2012}.

If $Y\simeq \PP^3$ (the case~\ref{th:min-degree}\ref{th:Enriques-P3}), then the morphism $\varphi : X \to \PP^3$
is a double cover with branch divisor
$B \subset \PP^3$ of degree $6$ by \eqref{eq:hyp-B}.
The groups $\Alt_7$ and $\PSp_4(\bF_3)$ have no non-trivial invariant hypersurfaces of degree 6. If $\bar G\simeq \PSL_2(\bF_7)$,
then we get Example~\xref{ex:sl27}\xref{ex:sl27:double6}.
Likewise, if $\bar G\simeq \Alt_6$, we get the case \xref{ex:3A6}\xref{ex:3A6-sextic-dc}.

If $Y$ is a cone over the Veronese surface (case~\ref{th:min-degree}\ref{th:Enriques-Veronese}), then
 $Y\simeq \PP(1,1,1,2)$ and $\OOO_{\PP}(B)=\OOO_{\PP}(6)$ by \eqref{eq:hyp-B}.
 Hence $\iota(X)=2$ and $X$ is a del Pezzo threefold of degree $1$.
This case was considered in Proposition~\ref{Proposition-iota=2}.

Finally consider the case where $Y$ is either a rational scroll,
a cone over a rational scroll, or a cone over a rational normal curve.
Then $Y$ is the image of $\hat Y:=\PP_{\PP^1} (\EEE)$, where $\EEE$ is a nef rank $n$ vector bundle on $\PP^1$,
under the map defined by the linear system $|\OOO(1)|$.
Thus $\nu:\hat Y\to Y$ is the ``minimal'' resolution of singularities
which is given by the blowup of the maximal ideal of $\Sing(Y)$.
In particular, $\nu$ is $G$-equivariant.
Then we have the following equivariant commutative diagram
\begin{equation*}
\xymatrix@R=15pt@C=50pt{
\hat X\ar[d]_{\eta}\ar[r]^{\hat\varphi} & \hat Y\ar[d]^{\nu}\ar@/^5.0pt/[dr]^{\pi}&
\\
X\ar[r]^{\varphi}& Y & \PP^1
}
\end{equation*}
where $\hat X$ is the normalization of the fiber product.
The morphism $\eta$ is a crepant contraction and $\hat X$ has at worst canonical Gorenstein singularities
\cite[Lemma~3.6]{Przhiyalkovskij-Cheltsov-Shramov-2005en}.
The group $\bar G$ trivially acts on $\PP^1$, so it non-trivially acts on each
fiber $F\simeq \PP^2$ of $\pi$.
Write $\EEE=\OOO_{\PP^1}(d_1)\oplus \OOO_{\PP^1}(d_2)\oplus \OOO_{\PP^1}(d_3)$, where $d_3\ge d_2\ge d_1\ge 0$.
If $d_1<d_2$, then the surjection $\EEE \to \OOO_{\PP^1}(d_2)\oplus \OOO_{\PP^1}(d_3)$
defines an invariant subscroll $\hat Y_2\subset \hat Y$ such that $Y_2\cap F$ is a line.
According to Theorem~\ref{th:GL3} this is impossible.
Thus we may assume that $d_1=d_2$ and, similarly, $d_2=d_3$.
So, $\hat Y\simeq \PP^2\times \PP^1\simeq Y$.
Note that $B$ intersects $F$ along a quartic curve
which must be invariant by \eqref{eq:hyp-B}.
Then the only possibility is that $B$ is a divisor of bidegree $(4,0)$ or a (reducible) divisor of bidegree $(4,2)$,
where $\bar G\simeq \PSL_2(\bF_7)$.
In the former case $X$ is the product of $\PP^1$ and del Pezzo surface of degree $2$.\
In the latter case $X$ is described in \cite[Example~1.8]{Krylov2016} as the threefold $X_1$.
In both cases the group $G$ splits by Lemma \ref{lemma:DI}, a contradiction.
\end{proof}

\begin{sremark}
Assume that $-K_X$ is very ample.
Then, by Proposition~\ref{prop:canLinearizable},
our group $G$ acts faithfully on the space $H^0(X,-K_X)^\vee$ so that
the induced action on its projectivization $\PP(H^0(X,-K_X)^\vee)=\PP^{g+1}$
is also faithful.
This implies that the representation $H^0(X,-K_X)^\vee$ of $G$ is reducible.
\end{sremark}

\begin{lemma}\label{lemma:trigonal}
Let $X$ be a Fano threefold with at worst
canonical Gorenstein singularities.
Assume that $\Aut(X)$ contains a subgroup $G$ as in the list \eqref{eq:list}.
Assume that the linear system $|-K_X |$ is very ample but the image $X=X_{2g-2}\subset \PP^{g+1}$
is not an intersection of quadrics.
Then $(X,G)$ is as in Example~\xref{ex:sl27}\xref{ex:sl27-quartic}.
\end{lemma}

\begin{proof}
By our assumption $g\ge 3$. If $g=3$, then
$X=X_4\subset \PP^4$ is a quartic. Inspecting the list \eqref{eq:list}
one can see that the only possibility is $G\simeq \SL_2(\bF_{7})$,
which implies that $(X,G)$ is as in Example~\xref{ex:sl27}\xref{ex:sl27:double6}.

Now assume that $g>3$.
Since $X=X_{2g-2}\subset \PP^{g+1}$ is projectively normal \cite{Iskovskikh-1980-Anticanonical},
the restriction map
\begin{equation*}
H^0(\PP^{g+1},\OOO_{\PP^{g+1}}(2)) \longrightarrow H^0(X,\OOO_X(2))
\end{equation*}
is surjective. This allows us to compute that the number of linear independent quadrics passing through
$X$ is equal to
\begin{equation*}
\textstyle \frac12(g - 2)(g - 3)>0.
\end{equation*}
Let $Y\subset \PP^{g+1}$ be the intersection of all quadrics containing $X$.
It is known that $Y$ is a reduced irreducible variety of minimal degree
(see \cite{Iskovskikh-1980-Anticanonical}
and \cite{Przhiyalkovskij-Cheltsov-Shramov-2005en}). Thus $Y$ is described by Theorem~\ref{th:min-degree}.

If $g=4$, then $Y$ is a (unique) quadric passing through $X$ and $X$
is cut out on $Y$ by a cubic, say $Z$. We may assume that $Z$ is $G$-invariant.
Thus our group $G$ has invariants of degrees $2$ and $3$.
Hence, $\z(G)$ is of order $6$ and so $G\supset 6.\Alt_6$.
But then $G$ has no faithful reducible representations of dimension $6$ by Table~\ref{table}.

Thus it remains to consider the case where $Y$ is either a rational scroll,
a cone over a rational scroll, or a cone over a rational normal curve.
Arguing as in the proof of Lemma~\ref{lemma:very-ample} and using \cite[Lemma~4.7]{Przhiyalkovskij-Cheltsov-Shramov-2005en},
we conclude that there exists a $G$-equivariant crepant extraction $\eta:\hat X\to X$ and a degree $3$ del Pezzo
fibration $\hat X\to \PP^1$. This gives us a contradiction (see Theorem~\ref{th:Cr2simple}).
\end{proof}

\begin{lemma}
\label{lemma:inv-hyp}
Let $X$ be a Fano threefold with at worst canonical singularities
and $G\subset \Aut(X)$ be a finite group contained in the list \eqref{eq:list}.
Assume that there exists a $G$-invariant
divisor $S\in |-K_X|$ such that
the pair $(X,S)$ is not plt and either $S$ is irreducible or
$(X,S)$ is not lc. Then
$G$ has a fixed point $P\in S$ such that $(X,S)$ is not plt at $P$.
\end{lemma}
\begin{proof}
Let $c$ be the log canonical threshold of $(X,S)$, that is,
the pair $(X,cS)$ is maximally lc. Then
$c \le 1$ and $-(K_X + cS)$ is nef.

Consider the case $c<1$. Then $-(K_X + cS)$ is ample.
Let $\Lambda\subset X$ be the locus of lc singularities of $(X,cS)$.
By Shokurov's connectedness principle
(see \cite{Shokurov-1992-e-ba} and \cite[ch. 17]{Utah}), $\Lambda$
is connected (and clearly $G$-invariant). If $\dim \Lambda=0$, then
$\Lambda$ must be an invariant point.
Suppose $\dim \Lambda=1$. If there exists a zero-dimensional center of lc singularities,
then replacing $cS$ with small invariant perturbation $(c-\epsilon)S+\Delta$
we get a zero-dimensional locus of lc singularities and may argue as above
(see \cite[Claim 4.7.1]{Prokhorov2012}).
Otherwise $\Lambda$ must be a minimal center of lc singularities and by Kawamata subadjunction theorem \cite[Theorem~1]{Kawamata-1998}
$\Lambda$ is a smooth rational curve, because $-(K_X + cS)$ is ample.
Since $G$ cannot act non-trivially on $\PP^1$,
the action of $G$ on $\Lambda$ is trivial.

Now consider the case $c=1$ and $S$ is irreducible. Then $(X,S)$ is lc.
Let $\nu : S' \to S$ be the normalization.
Write
\[
0 \sim \nu^* (K_X + S)|_ S = K_ {S'} + D',
\]
where $D'$ is the different, an effective integral
Weil divisor on $S'$ such that the pair $(S', D')$ is lc (see
\cite[\S 3]{Shokurov-1992-e-ba}, \cite[ch. 16]{Utah}, and \cite{Kawakita2007}).
The group $G$ acts naturally on $S'$ and $\nu$ is $G$-equivariant.
Now consider the minimal resolution
$\mu : \tilde S \to S'$ and let $\tilde D$ be a uniquely
defined) divisor such that
\begin{equation*}
K_{\tilde S}+\tilde D=\mu^* (K_{S'}+D')\sim 0,\qquad \mu_*\tilde D=D'.
\end{equation*}
is usually called the log crepant pull-back of $D'$. Here $\tilde D$
is again an effective reduced divisor.
Run $G$-equivariant MMP on $\tilde S$.
Clearly, the whole $\tilde D$ cannot be contracted.
We get a model $(S_{\min}, D_{\min})$ such that $(S_{\min}, D_{\min})$ is lc,
$(K_{S_{\min}}+D_{\min})\sim 0$, and $D_{\min}\neq 0$.
Assume that $S_{\min}$ has an equivariant conic bundle structure $\pi: S_{\min}\to B$.
Then $D_{\min}$ has one or two horizontal components which must be $G$-invariant.
By adjunction any horizontal component of $D_{\min}$ is either rational or elliptic curve.
Such a curve does not admit a non-trivial action of $G$, so the action of $G$ on
the corresponding component $\tilde D_1\subset \tilde D$ and
$\nu(\mu (\tilde D_1))$ must be trivial.
Similarly, if $S_{\min}$ is a del Pezzo surface with $\rk\Pic(S_{\min})^G=1$, then
$K_{S_{\min}}^2=2$ or $9$ by Theorem~\ref{th:Cr2simple} and so
$\tilde D$ has at most $3$ components. Arguing as above we get that the action of $G$ on
some component $\tilde D_1\subset \tilde D$ is trivial.
\end{proof}

\section{Proof of main result}
\label{section:G-rigid}

In this section, we prove Theorem~\ref{theorem:main} omitting the proof that $3.\Alt_7$ cannot act faithfully
on a rationally connected threefold. The case of $3.\Alt_7$ will be dealt with in Section~\ref{section:3A7} later.

\subsection{Singularities of quotients}
\label{subsection:quotients}

First, we need two auxiliary local results.

\begin{slemma}
\label{lemma:quotients-involuion}
Let $(X\ni P)$ be a threefold terminal singularity of index $1$.
Suppose that a group $A$ of order $2$ acts on $(X\ni P)$
so that either the action is free in codimension $1$ or the fixed point
locus is a $\QQ$-Cartier divisor.
Then the quotient $(X\ni P)/A$ is canonical.
\end{slemma}

\begin{proof}
Denote $(Y\ni Q):= (X\ni P)/A$.
If $A$ acts freely in codimension one, then the assertion is
a consequence of \cite[Proposition~6.12]{Kollar-ShB-1988}.
Let $\Fix(A,X)$ contain a divisor, say $D$.
We have an $A$-equivariant embedding $(X\ni P)\subset (\CC^4\ni 0)$
and we may assume that the action on $\CC^4$ is diagonalizable.
If this action is of type $\frac12 (1,0,0,0)$, then $\CC^4/A$
is smooth and so $Y$ is Gorenstein. Since the quotient singularities
are always rational \cite[Proposition~5.13]{Kollar-Mori-1988}, this implies that $(Y\ni Q)$ is canonical \cite[Corollary~5.24]{Kollar-Mori-1988}.
Thus we may assume that the action is of type $\frac12 (1,1,0,0)$
and the equation of $X$ is of the form
\begin{equation*}
\phi=x_1\phi_1(x_1,\dots,x_4)+ x_2\phi_2(x_1,\dots,x_4)=0.
\end{equation*}
But then the fixed point locus is not a $\QQ$-Cartier divisor.
\end{proof}

\begin{slemma}
\label{lemma:quotients-index-2}
Let $(X\ni P)$ be a threefold terminal cyclic quotient singularity of index $2$
acted by a finite group $G_P$ such that $\z(G_P)$ contains a subgroup $A\simeq \mumu_2$.
Suppose that the center of any extension of $G_P$ by $\mumu_2$ does not contain
an element of order $4$.
Then the quotient $(X\ni P)/A$ is canonical.
\end{slemma}
\begin{proof}
Let $\pi: (X^\sharp\ni P^\sharp)\to (X\ni P)$ be the index-one cover,
so that $(X\ni P)=(X^\sharp\ni P^\sharp)/\mumu_2$ where the action of $\mumu_2$
is free outside $P^\sharp$ \cite[Definition~5.19]{Kollar-Mori-1988}.
By our assumption $(X^\sharp\ni P^\sharp)\simeq (\CC^3,0)$
and the action of $A$ is of type $\frac12(1,1,1)$.
The action of $G_P$ lifts to an action of $G_P^\sharp$ on $(X^\sharp\ni P^\sharp)$, where
$G_P^\sharp$ is an extension of $G_P$ by $\mumu_2$.
Let $A^\sharp\subset G_P^\sharp$ be the preimage of $A$ (a group of order $4$).

If $A^\sharp\simeq \mumu_2\times \mumu_2$, then the elements of this group act as follows:
$\frac12(1,1,1)$, $\frac12(1,1,0)$, $\frac12(0,0,1)$.
It is easy to see that in this case the quotient is canonical Gorenstein.

Assume that the extension $A^\sharp\simeq \mumu_4$. Then
the action on $\CC^3$ is of type $\frac14(1,1,1)$ or $\frac14(1,1,-1)$.
By our assumption the former case does not occur.
In the latter case the quotient is terminal.
\end{proof}

\subsection{$G$-birationally superrigid Fano threefolds}
\label{subsection:G-rigid}

Second, we need the following global result.

\begin{theorem}
\label{theorem:Fano-super-rigid}\label{cases:groups}
Let $X$ be a Fano threefold with terminal Gorenstein singularities, and let $G$ be a finite subgroup in $\Aut(X)$.
Suppose that $X$ and $G$ fit one of the following seven cases:
\begin{enumerate}
\item \label{cases:groups:A7-23}
$G\simeq\Alt_7$, and $X$ is the unique smooth intersection of a quadric and a cubic in $\PP^5$ that admits a faithful action of the group $\Alt_7$;

\item \label{cases:groups:A7-P3}
$G\simeq\Alt_7$ and $X=\PP^3$;

\item\label{cases:groups:PSp43-P3}
$G\simeq\PSp_4(\bF_3)$ and $X=\PP^3$;

\item\label{cases:groups:PSp43-B}
$G\simeq\PSp_4(\bF_3)$ and $X$ is the Burkhardt quartic in $\PP^4$;

\item\label{cases:groups:SL2-11-K}
$G\simeq\PSL_2(\bF_{11})$ and $X$ is the Klein cubic threefold in $\PP^4$;

\item\label{cases:groups:SL2-11-8}
$G\simeq\PSL_2(\bF_{11})$ and $X$ is the unique smooth Fano threefold of Picard rank $1$ and genus $8$
that admits a faithful action of the group $\PSL_2(\bF_{11})$.
\end{enumerate}
Let $\rho\colon X\dasharrow V$ be a $G$-birational map such that $V$ is a Fano variety with at most canonical singularities.
Then $\rho$ is biregular.
\end{theorem}

The proof of this result is based on the following technical result, which originated in \cite{CheltsovShramov2012TG,CheltsovShramovTAMS,CheltsovShramov2015CRC,CheltsovShramov2017}.

\begin{proposition}
\label{proposition:Fano-super-rigid-technical}
Let $X$ be a Fano threefold with terminal Gorenstein singularities, and let $G$ be a finite subgroup in $\Aut(X)$.
Write $-K_{X}\sim nH$, where $H$ is a Cartier divisor on $X$, and $n=\iota(X)$ is the Fano index of the threefold $X$.
Let $\MMM$ be a linear system on the threefold $X$ that does not have fixed components,
and let $\lambda$ be a positive rational number such that
\begin{equation*}
\lambda\MMM\qsim -K_{X}.
\end{equation*}
Suppose that $(X,\lambda\MMM)$ does not have terminal singularities.
Then one of the following (non-exclusive) possibilities holds:
\begin{itemize}
\item there exists a $G$-orbit $\Sigma\subset X$ such that
\begin{equation*}
|\Sigma|=h^0\left(\OOO_X\big((n+1)H\big)\right)-h^0\left(\OOO_X\big((n+1)H\big)\otimes\III_{\Sigma}\right),
\end{equation*}
where $\III_{\Sigma}$ is the ideal sheaf of $\Sigma$;

\item there exists a $G$-irreducible reduced curve $C$ that consists of $r\geqslant 1$
pairwise disjoint smooth isomorphic irreducible components $C_1,\ldots,C_r$ such that
$2g-2\leqslant nd$, $rd\leqslant H^3n^2$ and
\begin{equation*}
r\left((n+1)d-g+1\right)=h^0\left(\OOO_X\big((n+1)H\big)\right)-h^0\left(\OOO_X\big((n+1)H\big)\otimes\III_{C}\right),
\end{equation*}
where $d=H\cdot C_i$, $g$ is the genus of any curve $C_i$, and $\III_{C}$ is the ideal sheaf of $C$.
\end{itemize}
\end{proposition}

\begin{proof}
Since $-K_X$ is Cartier, the log pair $(X,2\lambda\MMM)$ does not have klt singularities by \cite[Lemma~2.2]{CheltsovShramovTAMS}.
Choose $\mu\leqslant 2\lambda$ such that $(X,\mu\MMM)$ is strictly lc.
Let $Z$ be a minimal center of lc singularities of the log pair $(X,\mu\MMM)$,
see \cite{Kawamata1997,Kawamata-1998} for a precise definition.
Then $\theta(Z)$ is also a minimal center of lc singularities of
this log pair for every $\theta\in G$.
Moreover, we have
\begin{equation*}
Z\cap \theta(Z)\ne \emptyset\iff Z=\theta(Z)
\end{equation*}
by \cite[Proposition~1.5]{Kawamata1997}.

Since $\MMM$ does not have fixed components, the center $Z$ is either a curve or a point.
Observe that
\begin{equation*}
\mu\MMM\qsim \frac{\mu}{\lambda} nH,
\end{equation*}
where $\frac{\mu}{\lambda}\leqslant 2$.
It would be easier to work with $(X,\mu\MMM)$ if it did not have centers
of lc singularities that are different from $\theta(Z)$ for $\theta\in G$.
This is possible to achieve if we replace the boundary $\mu\MMM$ by (a slightly more complicated) effective boundary $B_X$ such that
\begin{equation*}
B_X\qsim \left(\frac{\mu}{\lambda}n+\epsilon\right)H
\end{equation*}
for some positive rational number $\epsilon$ that can be chosen arbitrary small.
This is known as the Kawamata--Shokurov trick or the perturbation trick (see \cite[Lemma~2.4.10]{CheltsovShramov2015CRC}
and the~proofs of \cite[Theorem~1.10]{Kawamata1997} and \cite[Theorem~1]{Kawamata-1998}).
By construction, we may assume that
\begin{equation} \label{eq:Bcondition}
\frac{\mu}{\lambda}n+\epsilon\leqslant 2n+\epsilon<2n+1.
\end{equation}
Note that the coefficients of $B_X$ depend on $\epsilon$.
But we can chose $\epsilon$ as small as we wish,
so that the number $\frac{\mu}{\lambda}n+\epsilon$ can be as close to $2n$ as we need.

Let $\Sigma$ be the union of all log canonical centers $\theta(Z)$ for $\theta\in G$.
Then $\Sigma$ is either a $G$-orbit or a disjoint union of irreducible isomorphic curves,
which are transitively permuted by $G$.
In both cases, we have an exact sequence of vector spaces
\begin{multline}
\label{eqation:exact-sequence}
0\longrightarrow H^0\left(\OOO_X\big((n+1)H\big)\otimes\III_{\Sigma}\right)
\longrightarrow H^0\left(\OOO_X\big((n+1)H\big)\right)\longrightarrow
\\
\longrightarrow H^0\left(\OOO_{\Sigma}\otimes\OOO_X\big((n+1)H\big)\right)\longrightarrow H^1\left(\OOO_X\big((n+1)H\big)\otimes\III_{\Sigma}\right),
\end{multline}
where $\III_{\Sigma}$ is an ideal sheaf of the locus $\Sigma$, and $\OOO_{\Sigma}$ is its structure sheaf.
Note that $\III_{\Sigma}$ is the multiplier ideal sheaf of the log pair $(X,B_X)$.
Since
\begin{equation*}
K_X+B_X\qsim \left(\left(\frac{\mu}{\lambda}-1\right)n+\epsilon\right)H,
\end{equation*}
we can apply Nadel's vanishing (see \cite[Theorem~9.4.17]{Lazarsfeld2004})
to deduce that
$$
h^1\left(\OOO_X((n+1)H)\otimes\III_{\Sigma}\right)=0.
$$
In particular, if $Z$ is a point, it follows from \eqref{eqation:exact-sequence} that
$$
|\Sigma|=h^0\left(\OOO_X\big((n+1)H\big)\right)-h^0\left(\OOO_X\big((n+1)H\big)\otimes\III_{\Sigma}\right).
$$

To complete the proof of the proposition, we may assume that $\Sigma$ is disjoint union of irreducible isomorphic curves
$C_1=Z,C_2,\ldots,C_r$, which are transitively permuted by $G$.
In particular, if $r=1$, then $\Sigma=C_1=Z$ is a $G$-invariant irreducible curve in $X$.

Let $d=H\cdot C_i$. Then $rd\leqslant H^3n^2$.
This immediately follows from Corti's \cite[Theorem~3.1]{Corti2000}.
Namely, observe that $(X,\mu\MMM)$ is not klt at general points of every curve $C_i$.
Let $M$ and $M^{\prime}$ be general surfaces in $\MMM$.
Then, applying \cite[Theorem~3.1]{Corti2000} to the log pair $(X,\mu\MMM)$ at general point of the curve $C_i$,
we obtain
\begin{equation*}
\mult_{C_i}\left(M\cdot M^{\prime}\right)\geqslant\frac{4}{\mu^2}
\end{equation*}
Then
\begin{equation*}
\frac{n^2}{\lambda^2}H^3=H\cdot M\cdot M^{\prime}\geqslant
\sum_{i=1}^rH\cdot C_i\mult_{C_i}\left(M\cdot M^{\prime}\right)\geqslant rd\frac{4}{\mu^2},
\end{equation*}
so that $rd\leqslant H^3n^2\frac{\mu^2}{4\lambda^2}\leqslant H^3n^2$ as claimed.

By Kawamata's subadjunction \cite[Theorem~1]{Kawamata-1998}, each curve $C_i$ is smooth.
Let $g$ be its genus.
Moreover, for every ample $\QQ$-Cartier $\QQ$-divisor $A$ on $X$, it
follows from \cite[Theorem~1]{Kawamata-1998} that
\begin{equation*}
\left(K_{X}+B_{X}+A\right)\Big\vert_{C_i}\qsim K_{C_i}+B_{C_i}
\end{equation*}
for some effective $\QQ$-divisor $B_{C_i}$ on the~curve $C_i$.
Computing the degrees of the left hand side and the right hand side in this $\QQ$-linear equivalence, we see that $2g-2\leqslant nd$.
In particular, the divisor $(n+1)H\vert_{C_i}$ is non-special on $C_i$, so that
\begin{equation*}
h^0\left(\OOO_{C_i}\big((n+1)H\vert_{C_i}\big)\right)=(n+1)d-g+1
\end{equation*}
by the Riemann--Roch formula. Now using \eqref{eqation:exact-sequence}, we get
$$
r\left((n+1)d-g+1\right)=h^0\left(\OOO_X\big((n+1)H\big)\right)-h^0\left(\OOO_X\big((n+1)H\big)\otimes\III_{\Sigma}\right),
$$
which complete the proof of the proposition.
\end{proof}

\begin{remark}
\label{remark:Fano-super-rigid-technical}
In the notations and assumptions of Proposition~\ref{proposition:Fano-super-rigid-technical},
there exists a central extension $\widetilde{G}$ of the group $G$ such that the line bundle $H$ is $\widetilde{G}$-linearizable.
Recall the Riemann-Roch theorem for a divisor $D$ on a smooth threefold $X$:
\begin{equation*}
 \chi(\OOO_X(D)) = \frac 16 D^3-\frac 14 D^2\cdot K_X +\frac1 {12} D\cdot K_X^2 +\frac1 {12} D\cdot c_2(X)+ \chi(\OOO_X).
\end{equation*}
Thus, the vector space $H^0(\OOO_X((n+1)H))$ is a representation of the group $\widetilde{G}$ of dimension
\begin{equation*}
h^0(\OOO_X((n+1)H))=\frac{(n+1)(2n+1)(3n+2)}{12}H^3+\frac{2n+2}{n}+1.
\end{equation*}
Then the exact sequence \eqref{eqation:exact-sequence} in the proof of Proposition~\ref{proposition:Fano-super-rigid-technical}
is an exact sequence of $\widetilde{G}$-representations.
\end{remark}

\begin{proof}[Proof of Theorem~\ref{theorem:Fano-super-rigid}]
Suppose that there exists a non-biregular $G$-birational map $\rho\colon X\dasharrow V$
such that $V$ is a Fano variety with at most canonical singularities.
Applying \cite[Theorem~3.2.1]{CheltsovShramov2015CRC}, we see that
there exists a $G$-invariant linear system $\MMM$ on the threefold $X$ such that $\MMM$ does not have fixed components,
and the singularities of the log pair $(X,\lambda\MMM)$ are not terminal,
where $\lambda$ is a positive rational number such that $\lambda\MMM\qsim -K_{X}$.
We will obtain a contradiction using Proposition~\ref{proposition:Fano-super-rigid-technical} and Remark~\ref{remark:Fano-super-rigid-technical}.

Let $m=\textstyle\frac{(n+1)(n+2)}{12}H^3+\frac{2}{n}$.
Note that the divisor $H$ is very ample in each of our cases, and $h^0(\OOO_X(H))=m+1$.
Thus, we may identify $X$ with its image in $\PP^m$.
Then $G$ is a subgroup in $\PGL_{m+1}(\CC)$.
Let $\widetilde{G}$ be a finite subgroup in $\GL_{m+1}(\CC)$ that maps surjectively to $G$ by the natural projection.
We may assume that $\widetilde{G}$ is the smallest group with this property.

The vector space $H^0(\OOO_X(H))$ is an irreducible representation of the group $\widetilde{G}$.
From Table~\ref{table}, this is immediate
in all cases except \xref{cases:groups}\xref{cases:groups:SL2-11-8}.
In this remaining case, $G\simeq\PSL_2(\bF_{11})$ and $X$ the unique smooth Fano threefold of Picard rank $1$ and genus $8$
that admits a faithful action of the group $\PSL_2(\bF_{11})$
(see Example~2.9~of~\cite{Prokhorov2012}).
The space $H^0(\OOO_X(H))$ is isomorphic to the representation $\bigwedge^2 V$
where $V$ is a $5$-dimensional faithful representation of $G$.
Recall that
\begin{equation*}
\chi_{\bigwedge^2V}(g) = \frac{1}{2}\left( \chi_V(g)^2 - \chi_V(g^2) \right)
\end{equation*}
where $g \in G$ and $\chi_{\bigwedge^2V}$ (resp.~$\chi_V$) is the
character of $\bigwedge^2V$ (resp.~$V$).
Evaluating $g$ at any element of order $2$ in
$\PSL_2(\bF_{11})$, we conclude that $\bigwedge^2V$ is irreducible
from the character table~\cite{atlas}.

Since $H^0(\OOO_X(H))$ is an irreducible representation of the group $\widetilde{G}$,
the threefold $X$ does not contain $G$-invariant subvarieties contained in a proper linear subspace of $\PP^m$.

Applying Proposition~\ref{proposition:Fano-super-rigid-technical},
we see that either $X$ contains a $G$-orbit $\Sigma$ such that
\begin{equation*}
m<|\Sigma|=h^0\left(\OOO_X\big((n+1)H\big)\right)-h^0\left(\OOO_X\big((n+1)H\big)\otimes\III_{\Sigma}\right),
\end{equation*}
where $\III_{\Sigma}$ is the ideal sheaf of $\Sigma$,
or there exists a $G$-irreducible reduced curve $C$ that is a disjoint
union of smooth irreducible curves $C_1,\ldots,C_r$ of genus $g$ and degree $d=H\cdot C_i$ such that
$rd\leqslant H^3n^2$, $2g-2\leqslant nd$ and
\begin{equation*}
r\left((n+1)d-g+1\right)=h^0\left(\OOO_X\big((n+1)H\big)\right)-h^0\left(\OOO_X\big((n+1)H\big)\otimes\III_{C}\right),
\end{equation*}
where $\III_{C}$ is the ideal sheaf of the curve $C$.
In the former case, by Remark~\ref{remark:Fano-super-rigid-technical},
the number $|\Sigma|$ is the dimension of some $\widetilde{G}$-subrepresentation in $H^0(\OOO_X((n+1)H))$.
Likewise, in the latter case, the number $r((n+1)d-g+1)$ is also the dimension of some $\widetilde{G}$-subrepresentation in $H^0(\OOO_X((n+1)H))$.
Moreover, if $r=1$, then the natural homomorphism $G\to\Aut(C)$ is injective,
because $C$ is not contained in a hyperplane in this case.
Thus, if $r=1$, then
\begin{equation}
\label{equation:Hurwitz}
84(g-1)\geqslant \lvert G\rvert
\end{equation}
by Hurwitz's automorphisms theorem.

\par\medskip\noindent
{\bf Case \xref{cases:groups}\xref{cases:groups:A7-23}.}
Here $G\simeq\Alt_7$ and $X$ is the unique smooth intersection of a quadric and a cubic in $\PP^5$ that admits a faithful action of the group $\Alt_7$.
We have $n=1$, $H^3=6$, $m=5$, $\widetilde{G}\simeq\Alt_7$ and $h^0(\OOO_X(2H))=20$.
Note that $H^0(\OOO_X(H))$ is the irreducible $\Alt_7$-representation
obtained as the quotient of the standard permutation representation by
the trivial representation.

Suppose that there exists a $G$-orbit $\Sigma$ in $X$ such that $5<|\Sigma|\leqslant 20$.
Using Table~\ref{table:maximal},
we see that $|\Sigma|$ is either $7$ or $15$.
In the case $|\Sigma|=15$, a stabilizer of a point in $\Sigma$ is isomorphic to
$\PSL_2(\bF_7)$.
The restriction of the representation $H^0(\OOO_X(H))$ to $\PSL_2(\bF_7)$
is the quotient of a transitive permutation representation so it has no
trivial subrepresentations.
Since $\PSL_2(\bF_7)$ is simple, it therefore has no one-dimensional
subrepresentations. The action cannot fix a point so this case is impossible.
In the case $|\Sigma|=7$, the stabilizer is isomorphic to $\Alt_6$.
In this case, there is a fixed point in the ambient space $\PP^5$
corresponding to the fixed point of the natural permutation action.
One checks that this point does not lie on $X$ by explicitly checking
the defining equations, which are just elementary symmetric functions.

Thus, the threefold $X$ contains a $G$-irreducible reduced curve $C$ that is a disjoint
union of smooth irreducible curves $C_1,\ldots,C_r$ of genus $g$ and degree $d$ such that
$rd\leqslant 6$ and $2g-2\leqslant d$.
As above, this shows that $r=1$, so that $d\leqslant 6$ and $g\leqslant 4$,
which contradicts \eqref{equation:Hurwitz}.

\par\medskip\noindent
{\bf Case \xref{cases:groups}\xref{cases:groups:A7-P3}.}
Here $G\simeq\Alt_7$ and $X=\PP^3$.
We have $n=4$, $H^3=1$, $m=3$, $\widetilde{G}\simeq 2.\Alt_7$ and $h^0(\OOO_X(5H))=56$.
Note that $H^0(\OOO_X(H))$ is an irreducible four-dimensional representation of the group $\widetilde{G}$.

Suppose that $\PP^3$ contains a $G$-orbit $\Sigma$ such that $3<|\Sigma|\leqslant 56$.
Going through the list of subgroups in $G$ of index $\leqslant 56$, we see that
\begin{equation*}
|\Sigma|\in\big\{7,15,21,35,42\big\}.
\end{equation*}
Let $G_P$ be the stabilizer of a point $P\in\Sigma$.
Using Table~\ref{table:maximal},
we conclude $G_P$ is isomorphic to one of the following groups:
$\Alt_6$, $\PSL_2(\bF_7)$, $\Sym_5$, $(\Alt_4\times\mumu_3)\rtimes\mumu_2$,
or $\Alt_5$.
Let $\widetilde{G}_P$ be a subgroup in $\widetilde{G}$ that is mapped to $G_P$.
We claim that the restriction of the representation $V=H^0(\OOO_X(H))$ to
$\widetilde{G}_P$ does not contain one-dimensional subrepresentations,
which contradicts the fact that $G_P$ fixes the point $P\in\PP^3$.
Let $\chi$ be the $4$-dimensional representation of $\widetilde{G}$;
we will consider restricted characters of $\chi$ (see \cite{atlas}).
We have $\chi(g)=-\frac{1}{2}(1\pm\sqrt{-7})$ when
$g \in \widetilde{G}$ has order $7$;
if $G_P \simeq \PSL_2(\bF_7)$ then the only possibility is that
$\widetilde{G}_P \simeq \PSL_2(\bF_7)$ and $V|_{\widetilde{G}_P}$ is irreducible.
We have $\chi(g)=-1$ when $g$ has order $5$, which means that
$V|_{\widetilde{G}_P}$ is irreducible if $G_P \cong \Alt_5$
(and a fortiori $\Sym_5$ and $\Alt_6$).
It remains to consider $G_P \cong (\Alt_4\times\mumu_3)\rtimes\mumu_2$,
which contains a $3$-Sylow subgroup $H \subseteq G$.
For some elements $g$ of order $3$, we have $\chi(g)=-2$
meaning that $\chi|_H$ does not have any trivial subrepresentations.
Since the group $G_P$ has no non-trivial maps to $\mumu_3$,
we conclude that there are no one-dimensional subrepresentations.
With the claim proved, we see this case is impossible.

Thus, there is a $G$-irreducible reduced curve $C$ in $\PP^3$ with the following properties:
$C$ is union of smooth irreducible curves $C_1,\ldots,C_r$ of genus $g$ and degree $d$,
$rd\leqslant 16$, $2g-2\leqslant 4d$ and
\begin{equation*}
r\big(5d-g+1\big)\leqslant 56.
\end{equation*}
As above, we see that $r\in\{1,7,15\}$.
If $r=15$, then $d=1$, so that $g=0$ and
\begin{equation*}
90=r\big(5d-g+1\big)\leqslant 56,
\end{equation*}
which is absurd. Likewise, if $r=7$,
then $d\leqslant 2$, so that $g=0$ and
\begin{equation*}
35d+7=7\big(5d+1\big)=r\big(5d-g+1\big)\leqslant 56,
\end{equation*}
so that $d=1$.
In this case, the stabilizer of the line $C_1$ is isomorphic to $\Alt_6$,
which is impossible, since the restriction of the representation $H^0(\OOO_X(H))$ to the subgroup $2.\Alt_6$ is irreducible.
Thus, we see that $r=1$, so that $C$ is irreducible.
Using \eqref{equation:Hurwitz}, we see that $g\in\{31,32,33\}$.
By Lemma~\ref{lem:Hurwitz}, one of the expressions
\[
\frac{169}{2^2\cdot 3 \cdot 7}, \quad
\frac{5071}{2^3\cdot 3^2 \cdot 5 \cdot 7}, \quad
\frac{634}{3^2 \cdot 5 \cdot 7}
\]
is a non-negative integer combination of expressions of the form
$1-\frac{1}{r}$ for $r=1,2,3,4,5,6,7,8,10,12,14$.
Since $\frac{1}{2}+\frac{2}{3}+\frac{6}{7}$ is larger than these
expressions, this is impossible.

\par\medskip\noindent
{\bf Case \xref{cases:groups}\xref{cases:groups:PSp43-P3}.}
Here $G\simeq\PSp_4(\bF_3)$ and $X=\PP^3$.
We have $n=4$, $H^3=1$, $m=3$, $\widetilde{G}\simeq \Sp_4(\bF_3)$ and $h^0(\OOO_X(5H))=56$.
We will see that $H^0(\OOO_X(5H))$ is a direct sum of irreducible
representations of $\widetilde{G}$ of dimensions $20$ and $36$.
Indeed, $h^0(\OOO_X(H))$ is an irreducible $4$-dimensional
representation $V$ of $\widetilde{G}$ with character $\chi$.
Note that every summand of $S^5V$ must be a faithful representation
of $\widetilde{G}$ since $5$ is coprime to $2$.
Via the Newton identities, we have the standard formula for the $5$th
symmetric power
\begin{multline*}
S^5\chi(g) =
\frac{1}{120}\Big[\chi(g)^5 + 10\chi(g^2)\chi(g)^3 +
15\chi(g^2)^2\chi(g)\\
+ 20\chi(g^3)\chi(g)^2 + 20\chi(g^3)\chi(g^2) + 30\chi(g^4)\chi(g)
+ 24\chi(g^5)\Big]
\end{multline*}
where $g$ is an element of $\widetilde{G}$.
From the character table \cite{atlas}, we see that $\chi(g)=-1$ for any
element $g$ of order $5$. We compute that $S^5\chi(g)=1$ and conclude
from the character table that the only possibility is a sum of
characters of degree $20$ and $36$ as desired.

Suppose that $\PP^3$ contains a $G$-orbit $\Sigma$ such that $|\Sigma|$
is the dimension of some $\widetilde{G}$-subrepresentation in $H^0(\OOO_X(5H))$.
Then
\begin{equation*}
|\Sigma|\in\{20,36,56\}.
\end{equation*}
Using Table~\ref{table:maximal}, we see that $|\Sigma|=36$.
Let $G_P$ be the stabilizer of a point $P\in\Sigma$.
Then $G_P\simeq\Sym_6$. The group $G$ contains one such subgroup up to conjugation.
Let $\widetilde{G}_P$ be a subgroup in $\widetilde{G}$ that is mapped to $G_P$.
Then $\widetilde{G}_P\simeq2.\Sym_6$,
and the restriction of the representation $H^0(\OOO_X(H))$ to
$\widetilde{G}_P$ does not contain one-dimensional subrepresentations.
This contradicts the fact that $G_P$ fixes the point $P\in\PP^3$.

Thus, there is a $G$-irreducible reduced curve $C$ in $\PP^3$ that is a union of smooth irreducible curves $C_1,\ldots,C_r$ of genus $g$ and degree $d$ such that $rd\leqslant 16$, $2g-2\leqslant 4d$ and
\begin{equation*}
r\big(5d-g+1\big)\leqslant 56.
\end{equation*}
Arguing as above, we see that $r=1$, so that $d\leqslant 16$ and $g\leqslant 33$,
which is impossible by \eqref{equation:Hurwitz}.

\par\medskip\noindent
{\bf Case \xref{cases:groups}\xref{cases:groups:PSp43-B}.}
Here $G\simeq\PSp_4(\bF_3)$ and $X$ is the Burkhardt quartic in $\PP^4$.
We have $n=1$, $H^3=4$, $m=4$, $\widetilde{G}\simeq G$ and $h^0(\OOO_X(2H))=15$.
Note that $\PP^3$ does not contain a $G$-orbit $\Sigma$ such that $4<|\Sigma|\leqslant 15$,
because $G$ does not contain subgroups of such index.
Thus, there is a $G$-irreducible reduced curve $C$ in $\PP^3$ that is union of smooth irreducible curves $C_1,\ldots,C_r$ of genus $g$ and degree $d$
such that $rd\leqslant 4$ and $2g-2\leqslant d$.
Since there are no subgroups of index $2$, $3$, or $4$, we have $r=1$.
Thus $d\leqslant 4$ and $g\leqslant 3$, which contradicts \eqref{equation:Hurwitz}.

\par\medskip\noindent
{\bf Case \xref{cases:groups}\xref{cases:groups:SL2-11-K}.}
Here $G\simeq\PSL_2(\bF_{11})$ and $X$ is the Klein cubic threefold in $\PP^4$.
We have $n=2$, $H^3=3$, $m=4$, $\widetilde{G}\simeq G$ and $h^0(\OOO_X(3H))=34$.
If $\chi$ is the character of the representation $V=h^0(\OOO_X(3H))$ of $G$,
then the character of the third symmetric power is given by
\begin{equation*}
S^3\chi(g) =
\frac{1}{6}\Big[\chi(g)^3 + 3\chi(g^2)\chi(g) + 2\chi(g^3) \Big]
\end{equation*}
for $g \in G$.
The trivial character occurs exactly once in $S^3V$ by Table~\ref{table}.
From the character table~\cite{atlas},
for any element $g \in G$ of order $3$ we have $\chi(g)=-1$ and thus
$S^3\chi(g)=2$.
Note that all irreducible characters $\rho$ satisfy $\rho(g) \in \RR$,
but only the trivial and $10$-dimensional ones have $\rho(g) > 0$.
Since there is only one trivial subrepresentation,
this forces the existence of at least one 10-dimensional irreducible
subrepresentation.
The possible irreducible characters have degrees $1,5,10,11,12$,
thus $1+10+12+12$ is the only possibility for $S^3V$.
For any $h \in G$ of order $5$, $\chi(h)=0$ and $S^3\chi(h)=0$
while $\rho(h) \ne 0$ for the $12$-dimensional irreducible
representations.
We conclude that the vector space $H^0(\OOO_X(3H))$ splits
as a sum of two non-isomorphic twelve-dimensional representations,
and one ten-dimensional representation.

Suppose that $X$ contains a $G$-orbit $\Sigma$ such that $|\Sigma|$
is the dimension of some $G$-subrepresentation in $H^0(\OOO_X(3H))$.
Going through the list of subgroups in $G$ of index $\leqslant 34$, we see that $|\Sigma|=12$.
Let $G_P$ be the stabilizer of a point $P\in\Sigma$.
Then $G_P\simeq\mumu_{11}\rtimes\mumu_5$,
and the restriction of the representation $H^0(\OOO_X(H))$ to $G_P$ is irreducible.
This contradicts to the fact that $G_P$ fixes a point in $\PP^4$.

Thus, there is a $G$-irreducible reduced curve $C$ in $\PP^4$ that is union of smooth irreducible curves $C_1,\ldots,C_r$ of genus $g$ and degree $d$
such that $rd\leqslant 12$, $2g-2\leqslant 2d$ and $r\big(3d-g+1\big)\leqslant 34$.
Going through the list of subgroups in $G$ of index $\leqslant 12$, we see that $r\in\{1,11,12\}$.
If $r=12$ or $r=11$, then $d=1$, so that $g=0$, which gives
\begin{equation*}
11\times 4\leqslant 4r=r\big(3d+1\big)=r\big(3d-g+1\big)\leqslant 34,
\end{equation*}
which is absurd. Thus, we have $r=1$, so that $C$ is irreducible.
Then $d\leqslant 12$ and $g\leqslant 13$.
Using \eqref{equation:Hurwitz}, we see that $g\geqslant 9$, so that $g\in\{9,10,11,12,13\}$.
Since $\lvert G\rvert=25920=2^6\cdot 3^4 \cdot 5$ and $2g-2$ is never divisible by
$2^5$ or $3^3$, the expression on the left hand side from
Lemma~\ref{lem:Hurwitz} has $2^2 \cdot 3^2$ in the denominator when
written in lowest terms.
However, it is a non-negative integer combination of expressions of the form
$1-\frac{1}{r}$ for $r=1,2,3,4,5,6,9,12$.
Thus, this case is impossible.

\par\medskip\noindent
{\bf Case \xref{cases:groups}\xref{cases:groups:SL2-11-8}.}
Here $G\simeq\PSL_2(\bF_{11})$ and $X$ is the unique smooth Fano threefold of Picard rank $1$ and genus $8$
that admits a faithful action of the group $\PSL_2(\bF_{11})$.
We have $n=1$, $H^3=14$, $m=9$, $\widetilde{G}\simeq G$ and $h^0(\OOO_X(3H))=40$.
Going through the list of subgroups in $G$ of index $\leqslant 40$,
we see that either $G_P\simeq\Alt_5$ or $G_P\simeq\mumu_{11}\rtimes\mumu_5$.
Since $X$ is smooth, the embedded tangent space at $P$ is $3$-dimensional and
has a faithful $G_P$ action.
This means that the restriction of the representation $H^0(\OOO_X(H))$ to $G_P$
has a trivial representation and a $3$-dimensional faithful subrepresentation.
This is impossible if $G_P\simeq\mumu_{11}\rtimes\mumu_5$,
so $G_P \simeq \Alt_5$.
From the character table~\cite{atlas} of $\PSL_2(\bF_{11}$,
we see that any faithful irreducible character $\chi$ of degree $\le 10$
must have $\chi(g)=0$ for all $g \in G$ of order $5$.
However, we have $\rho(g) = \frac{1}{2}(1\pm \sqrt{5})$ for an irreducible
character $\rho$ of $\Alt_5$ of degree $3$.
Thus there must be two $3$-dimensional faithful $\Alt_5$-subrepresentations of
$H^0(\OOO_X(H))$ along with a trivial subrepresentation.
The remaining character must be a character $\sigma$ of $\Alt_5$ of
degree $2$ such that $\sigma(g)=-2$.
No such characters exist, so this case is impossible.

Thus, there is a $G$-irreducible reduced curve $C$ in $X$ that is a union of smooth irreducible curves $C_1,\ldots,C_r$ of genus $g$ and degree $d$
such that $rd\leqslant 14$, $2g-2\leqslant d$ and $r(2d-g+1)\leqslant 40$.
Arguing as above, we see that $r\in\{1,11,12\}$.
As above, denote by $G_1$ the stabilizer of the curve $C_1$.
If $r=12$, then $d=1$, so that $C_1$ is a line,
which implies that the restriction of the representation $H^0(\OOO_X(H))$ to $G_1$ contains a two-dimensional subrepresentation.
But we already checked that this is not the case, so that $r\ne 12$.
Similarly, we see that $r\ne 11$, because $G_1\simeq\Alt_5$ in this case, and
the restriction of the representation $H^0(\OOO_X(H))$ to $G_1$
does not have two-dimensional subrepresentations either.
Hence, we see that $r=1$, so that $C$ is irreducible.
Then $d\leqslant 14$ and $g\leqslant 8$, which is impossible by \eqref{equation:Hurwitz}.
This completes the proof of Theorem~\ref{theorem:Fano-super-rigid}.
\end{proof}

\subsection{The proof}
\label{subsection:proof}

Now we are ready to prove

\begin{proposition}\label{prop:not-A7}
Let $X$ be a rationally connected threefold.
Then the group $\Aut(X)$ does not contain a subgroup isomorphic to
$\SL_2(\bF_{11})$, $\Sp_4(\bF_{3})$, $2.\Alt_7$ or $6.\Alt_7$.
\end{proposition}

\begin{proof}
Let $X$ be a rationally connected threefold.
Let $G$ be a subgroup in $\Aut(X)$.
Suppose that $G$ is one of the following groups $\SL_2(\bF_{11})$, $\Sp_4(\bF_{3})$, $2.\Alt_7$ or $6.\Alt_7$.
We seek a contradiction.
We may assume that $G\not\simeq 6.\Alt_7$ because in this case taking
the quotient by a subgroup of order $3$ in the center
reduces the problem to $2.\Alt_7$.
Thus $\z(G)\simeq\mumu_2$.
We may assume that $X$ has a structure of $G$-Mori fiber space $\pi\colon X\to S$.
By Theorem~\ref{th:Cr2simple} the base $S$ is a point, i.e.~$X$ is a $G\QQ$-Fano threefold.

Let $Y=X/\z(G)$, let $\pi: X\to Y$ be the quotient map and let $\bar G:=G/\z(G)$.
Then $Y$ has canonical singularities by Lemmas~\ref{lemma:quotients-involuion}, \ref{lemma:quotients-index-2} and Claim~\ref{claim:c} below.
Using Theorem~\ref{theorem:Prokhorov-2}, we see that $Y$ is $\bar G$-birational to one of the Fano threefolds listed in Theorem~\ref{theorem:Fano-super-rigid}
and by this theorem $Y$, in fact, is $\bar G$-isomorphic to one of the varieties in the list.
In all our cases $Y$ is a $\bar G$-Fano
with at worst isolated ordinary double points (in fact, $Y$ is smooth except for the case
\ref{cases:groups}\ref{cases:groups:PSp43-B}).
Moreover, $\Cl(Y)^{\bar G}=\Pic(Y)$ by \cite[Lemma~2.2]{CPS-BurBar}.
The Hurwitz formula gives
\begin{equation}
\label{eq:Hurwitz}
 K_X=\pi^*\left(K_Y+\textstyle\frac12 B\right),
\end{equation}
where $B$ is the branch divisor.
Thus $B$ is non-zero $\bar G$-invariant and there exists a Cartier divisor $D$ such that $2D\sim B$.
In particular, the Fano index of $Y$ is even.
Therefore, we are left with the cases
\ref{cases:groups}\ref{cases:groups:A7-P3}, \ref{cases:groups}\ref{cases:groups:PSp43-P3}, \ref{cases:groups}\ref{cases:groups:SL2-11-K}.
But in the case \ref{cases:groups}\ref{cases:groups:A7-P3} we have $Y\simeq \PP^3$, $\bar G\simeq \Alt_7$
and $\deg B\le 6$ by \eqref{eq:Hurwitz}
 which impossible because the minimal degree of invariants in this case
is at least $8$ (see Table~\ref{table}). Likewise we obtain a contradiction in the cases
\ref{cases:groups}\ref{cases:groups:PSp43-P3} and \ref{cases:groups}\ref{cases:groups:SL2-11-K}.
\end{proof}

\begin{sclaim}
\label{claim:c}
Any $\z(G)$-fixed point of $X$ satisfies conditions of
Lemma~\xref{lemma:quotients-involuion} or Lemma~\xref{lemma:quotients-index-2}.
\end{sclaim}

\begin{proof}
Note that the fixed point locus of $\z(G)$ on $X$ is $G$-invariant and so its divisorial part must be a $\QQ$-Cartier divisor.
Hence any Gorenstein point of $X$ satisfies conditions of Lemma~\ref{lemma:quotients-involuion}.
Assume that $X$ is a non-Gorenstein Fano threefold.
 Let $P\in X$ be a non-Gorenstein point
and let $G_P\subset G$ be its stabilizer.

Arguing as in the proof of \cite[Lemma~6.1]{Prokhorov2012} or Lemma~\ref{lemma:1frac2} below
one can show that $P\in X$ is a cyclic quotient
singularity of type $\frac12(1,1,1)$.
As in Lemma~\xref{lemma:quotients-index-2}, consider the index-one cover
$\pi: (X^\sharp\ni P^\sharp)\to (X\ni P)$
and the lifting $G_P^\sharp$ to $\Aut(X^\sharp\ni P^\sharp)$.
Since the length of the orbit of $P$ is at most $15$
(see \cite[Lemma~6.1]{Prokhorov2012} or Lemma~\ref{lemma:1frac2}), we
see that there are only the following possibilities \cite{atlas}, \cite{GAP4}:
\begin{enumerate}
\item\label{case:quotients-involuion-1}
$G\simeq 2.\Alt_7$, $G_P\simeq 2.\Alt_6$;
\item\label{case:quotients-involuion-2}
$G\simeq 2.\Alt_7$, $G_P\simeq \SL_2(\bF_7)$;
\item\label{case:quotients-involuion-3}
$G\simeq \SL_{2}(\bF_{11})$, $G_P\simeq 2.\Alt_5$;
\item\label{case:quotients-involuion-4}
$G\simeq \SL_{2}(\bF_{11})$, $G_P\simeq \mumu_{22}\rtimes \mumu_{5}$.
\end{enumerate}
Now one can see that in the cases \ref{case:quotients-involuion-1}-\ref{case:quotients-involuion-3} any extension $G_P^\sharp$ of
$G_P$ by $\mumu_2$ splits. Consider the case \ref{case:quotients-involuion-4}.
Assume that the center of $G_P^\sharp$ contains an element $z$ of order $4$.
Then the kernel of the homomorphism $G_P^\sharp\to \mumu_{5}$ must be a cyclic group $\mumu_{44}$.
But then $G_P^\sharp$ has no faithful $3$-dimensional representation, a contradiction.
\end{proof}

\section{$3.\Alt_7$}\label{section:3A7}
The aim of this section is to prove the following.
\begin{proposition}\label{prop:A7}
Let $X$ be a rationally connected threefold.
Then the group $\Aut(X)$ does not contain a subgroup isomorphic to $3.\Alt_7$.
\end{proposition}
Let $G=3.\Alt_7$.
Assume that $G\subset \Aut(X)$ where $X$ is a rationally connected threefold.
We may assume that $X$ has the structure of a $G$-Mori fiber space
$\pi: X\to S$. By Theorem~\ref{th:Cr2simple}, the base $S$ is a point, i.e.~$X$ is a $G\QQ$-Fano threefold.
We distinguish two cases: \ref{subsection:Gorenstein} and~\ref{subsection:non-Gorenstein}.

\subsection{Actions on Gorenstein Fano threefolds}
\label{subsection:Gorenstein}
First we consider the case where $K_X$ is Cartier, i.e.~the singularities of $X$ are
at worst terminal Gorenstein.
By Propositions~\ref{Proposition-rho>1} and~\ref{Proposition-iota=2} we have
$\Pic(X)=\ZZ\cdot K_X$. Let $g=\g(X)$ be the genus of $X$.
Thus $(-K_X)^3=2g-2$.
By Lemmas~\ref{lemma:base-points} and~\ref{lemma:very-ample}
the linear system $|-K_X|$ defines an embedding to $\PP^{g+1}$.
By \cite{Prokhorov-v22} we have $g\neq 12$.
Recall that any Fano threefold $X$ with terminal Gorenstein singularities admits
a smoothing, i.e. a deformation $\mathfrak X\to \mathfrak D\ni 0$ over a disk $\mathfrak D\subset \CC$ such that
the central fiber $\mathfrak X_0$ is isomorphic to $X$ and a general fiber is smooth \cite{Namikawa-1997}.
The numerical invariants such as the degree, the Picard number, and the Fano index are constant in
such a family $\mathfrak X/\mathfrak D$.
Now by the classification of
smooth Fano threefolds
\cite{Iskovskikh-1980-Anticanonical} or \cite{IP99}
we conclude that $g\le 10$.
By Lemma~\ref{lemma:fixed-point} the group $G$ has no fixed points.

\begin{sclaim}
 $X$ has no $G$-invariant hyperplane sections.
\end{sclaim}

\begin{proof}
Assume that there exists a $G$-invariant
divisor $S\in |-K_X|$. By Theorem~\ref{th:actions-K3} the pair $(X,S)$ is not plt.
By Lemma~\ref{lemma:inv-hyp} the surface $S$ is reducible and reduced.
Since $\Pic(X)=\ZZ\cdot [S]$, the linear system $|-K_X|=|S|$ has no fixed components.
By Bertini's theorem this $G$-invariant surface $S\in |-K_X|$ is unique.
Write $S=\sum _{i=1}^m S_i$.
Then
\[
\sum (-K_X)^2\cdot S_i = 2g-2\le 18.
\]
So the cardinality of the orbit of $S_1$ equals $7$ or $14$ by
Table~\ref{table:maximal}.
In both cases $g=8$ so that $\dim H^0(X,\OOO_X(-K_X))=10$.
We get a contradiction with Table~\ref{table}.
\end{proof}

Let $V:=H^0(X,\OOO_X(-K_X))^\vee$.
The group $G$ faithfully acts on $V$ and $\PP(V)$.
Hence the representation $V$ of $G$ is reducible.
Since $\dim V=g+2\le 12$, we have $V= V'\oplus V''$ (as a $G$-module), where $V'$ and $V''$ are irreducible
representations with $\dim V''=\dim V'=6$. Thus $g=10$.

\begin{slemma}\label{lemma:6dim}
Let $U$ be an irreducible $6$-dimensional representation of $3.\Alt_7$
and let $Q\subset S^2U^\vee$ be a $6$-dimensional subrepresentation.
Then the base locus of $Q$ on $\PP(U)$ is empty.
\end{slemma}

\begin{proof}
Observe from Table~\ref{table}, that all $6$-dimensional irreducible
representations $U$ have a unique invariant cubic in $\PP(U)$
defined by a polynomial $f$.
Then the quadratic polynomials $\partial f/\partial x_i$ generate
a $6$-dimensional irreducible subrepresentation
$Q \subseteq S^2U^\vee$ which is isomorphic to $U$.
The complement $U'$ of $Q$ in $S^2U^\vee$ is $15$-dimensional.
If $U$ is faithful, then so must be $U'$ and we conclude immediately
from the character table~\cite{atlas} that
$S^2U^\vee\simeq U\oplus U_{15}$, where $U_{15}$
is a faithful irreducible representation of dimension $15$.
If $U$ is not faithful, then from Table~\ref{table} we
see that $U'$ contains a unique trivial subrepresentation.
In this case,
$S^2U^\vee\simeq \CC\oplus U\oplus W_{14}$, where $W_{14}$
is a non-faithful irreducible representation of dimension $14$.
From the character table, the only possibility is that $U'$
is a direct sum of an irreducible $14$-dimensional representation
and a trivial representation.
In either case, $Q$ is the the unique $6$-dimensional irreducible
subrepresentation.
One can check that the hypersurface $f=0$ in $\PP(U)$ is smooth.
Hence quadrics from $Q$ have no common zeros in $\PP(U)$.
\end{proof}

\begin{sremark}
\label{remark:KE-cubic}
Let $X_3$ a cubic fourfold in $\PP^5$ that admits a faithful action of the group $\Alt_7$.
Then $X_3$ is one of two hypersurfaces that (implicitly) appear in the proof of Lemma~\ref{lemma:6dim}.
For one of them, the action of $\Alt_7$ is given by the standard irreducible six-dimensional representation of the group $\Alt_7$.
For the other, the action is given by an irreducible faithful six-dimensional representation of the group $3.\Alt_7$.
In the former case, one has $\alpha_{\Alt_7}(X_3)\leqslant\frac{2}{3}$, because $\PP^4$ contains an $\Alt_7$-invariant quadric hypersurface.
Here, $\alpha_{\Alt_7}(X_3)$ is the $\Alt_7$-invariant $\alpha$-invariant of Tian defined in \cite{Tian}.
However, in the latter case, one has $\alpha_{\Alt_7}(X_3)\geqslant 1$.
Indeed, suppose that $\alpha_{\Alt_7}(X_3)<1$.
Then there exists an effective $\Alt_7$-invariant divisor $D$ on $X_3$ such that $D\qsim -K_{X_3}$ and $(X_3,D)$ is not lc.
This follows from the algebraic formula for $\alpha_{\Alt_7}(X_3)$ given in \cite[Appendix~A]{Cheltsov-Shramov}.
Choose positive rational number $\mu<1$ such that $(X_3,\mu D)$ is strictly lc.
Let $Z$ be a minimal center of lc singularities of the log pair $(X_3,\mu D)$.
Then $\dim(Z)\leqslant 2$, because $\PP^6$ does not contain $\Alt_7$-invariant hyperplanes and quadric hypersurfaces.
Using the perturbation trick
(see the~proofs of \cite[Theorem~1.10]{Kawamata1997} and \cite[Theorem~1]{Kawamata-1998}),
we may assume that all log canonical centers of the log pair $(X_3,\mu D)$ are of the form $\theta(Z)$ for some $\theta\in G$.
Moreover, by \cite[Proposition~1.5]{Kawamata1997}, either $Z\cap\theta(Z)=\emptyset$ or $Z=\theta(Z)$ for every $\theta\in G$.
On the other hand, it follows from \cite[Theorem~9.4.17]{Lazarsfeld2004} that the union of all log canonical centers of the log pair $(X_3,\mu D)$
is connected, which implies that $Z$ is $\Alt_7$-invariant.
In particular, since $\PP^6$ does not have $\Alt_7$-fixed points, the center $Z$ is not a point,
and $\Alt_7$ acts faithfully on $Z$.
However, Kawamata's subadjunction \cite[Theorem~1]{Kawamata-1998} implies that $Z$ is a normal Fano type subvariety,
so that either it is a smooth rational curve or a rational surface with at most quotient singularities.
This is impossible, because $\Alt_7$ is not contained in $\Cr_2(\CC)$ by \cite{DolgachevIskovskikh}.
Thus, we see that $\alpha_{\Alt_7}(X_3)\geqslant 1$ in the case when
the action of $\Alt_7$ on $\PP^6$ is given by an irreducible faithful six-dimensional representation of the group $3.\Alt_7$.
In particular, this hypersurface admits a K\"ahler--Einstein metric by \cite{Tian}.
Note that there are other smooth cubic fourfolds that are known to be K\"ahler--Einstein.
They are described in \cite{Arezzo}.
Our $\Alt_7$-invariant cubic fourfold is not one of them:
in appropriate homogeneous coordinates on $\PP^6$ it is given by
\begin{multline*}
x_{1}^3+x_{2}^3+x_{3}^3+x_{4}^3+x_{5}^3+x_{6}^3+x_{1}x_{2}x_{3}+x_{1}x_{2}x_{4}+\\
+x_{1}x_{2}x_{5}+x_{1}x_{4}x_{5}+x_{1}x_{5}x_{6}+x_{2}x_{4}x_{5}+x_{2}x_{4}x_{6}+x_{3}x_{4}x_{5}+\\
\omega^2x_{1}x_{3}x_{4}+\omega^2x_{1}x_{4}x_{6}+\omega^2x_{2}x_{3}x_{5}+\omega^2x_{2}x_{5}x_{6}+\omega x_{1}x_{2}x_{6}+\omega x_{1}x_{3}x_{5}+\\
+\omega x_{1}x_{3}x_{6}+\omega x_{2}x_{3}x_{4}+\omega x_{2}x_{3}x_{6}+\omega x_{3}x_{4}x_{6}+\omega x_{3}x_{5}x_{6}+\omega x_{4}x_{5}x_{6}=0,
\end{multline*}
where $\omega$ is a primitive cubic root of unity.
This implies that it does not contain planes, while K\"ahler--Einstein
smooth cubic fourfolds found in \cite{Arezzo} always contain many planes.
\end{sremark}

Let $Q:=H^0(X,\JJJ_X(2))\subset S^2V^\vee$ be the space of quadrics passing through $X$. Then $\dim Q=28$.
Consider the decomposition
\begin{equation}
\label{eq:A7:S2}
S^2V=(S^2V') \oplus (S^2 V'') \oplus (V'\otimes V'').
\end{equation}
Since $X\not\supset \PP(V')$ and $X\not\supset \PP(V'')$ we have
\begin{equation*}
Q\cap S^2V'\neq 0, \qquad Q\cap S^2V''\neq 0.
\end{equation*}
If both representations $V'$ and $V''$ are faithful, then $V'\not \simeq V''$ and \eqref{eq:A7:S2}
has the form
\begin{equation*}
\begin{array}{cccccc}
S^2V=&S^2V' &\oplus& S^2 V''& \oplus& V'\otimes V''
\\
&\lsimeq&&\lsimeq&&\lsimeq
\\
& V''\oplus U_{15}
&& V'\oplus U_{15}'
&&\CC\oplus W_{14}\oplus W_{21}
\end{array}
\end{equation*}
where $U_{15}$, $U_{15}'$ are different faithful $15$-dimensional irreducible representations, and
$W_{k}$, $k=14$, $21$ are irreducible representations of $\Alt_7$ with $\dim W_k=k$.
The symmetric powers follow from the proof of Lemma~\ref{lemma:6dim}.
Since $V'$ and $V''$ are dual, the trace shows there is a trivial
subrepresentation of $V'\otimes V''$.
If $\chi$ is the character of the product $V'\otimes V''$,
we observe that $\chi(g)=0$ for all non-central $g$ of order $3$ and $4$.
Since the trivial representation is already a summand,
we require (non-faithful) subrepresentations with characters whose
values are not $\ge 0$ for such $g$.
Looking at the character table
$V'\otimes V'' \simeq \CC\oplus W_{14}\oplus W_{21}$
is the only possibility.

If $V''$ is not faithful, then similarly \eqref{eq:A7:S2}
has the form
\begin{equation*}
\begin{array}{cccccc}
S^2V=&S^2V' &\oplus& S^2 V''& \oplus& V'\otimes V''
\\
&\lsimeq&&\lsimeq&&\lsimeq
\\
& U_6\oplus U_{15}
&&\CC\oplus W_6 \oplus W_{14}
&&U_{15}'\oplus U_{21}
\end{array}
\end{equation*}
In this case, that $V'\otimes V'' \simeq U_{15}'\oplus U_{21}$
is the only possibility can be seen by considering the values of
faithful characters on the trivial element and an involution.

In both these decompositions all the irreducible summands are pairwise non-isomorphic.
Counting dimensions one can see that either $Q\cap S^2V'$ or $Q\cap S^2V''$ contains a $6$-dimensional subrepresentation.
Suppose that this holds, for example, for $Q\cap S^2V'$.
This implies that $X\subset \PP(V'')$, a contradiction.

\subsection{Actions on non-Gorenstein Fano threefolds}
\label{subsection:non-Gorenstein}
Now we consider the case where $K_X$ is not Cartier.
\begin{slemma}[{\cite{Kawamata-1992bF}}, {\cite{Reid-YPG1987}}]
Let $X$ be a (terminal) $\QQ$-Fano threefold whose non-Gorenstein singularities
are exactly $N$ cyclic quotient points of type $\frac12(1,1,1)$.
Then we have
\begin{eqnarray}
-K_X\cdot c_2&=& \textstyle 24 - \frac{3N}2,
\\
\dim |-K_X| &=&\textstyle \frac12(-K_X)^3 -\frac14N + 2,
\\
\dim |-2K_X|&=&\textstyle \frac52(-K_X)^3 -\frac14N + 4.
\end{eqnarray}
\end{slemma}

\begin{assumption}
Let $G=3.\Alt_7$ and let $X$ be a non-Gorenstein $G\QQ$-Fano threefold.
Let $\Sing'(X)$ be the set of non-Gorenstein points and let $N$ be its cardinality.
\end{assumption}

\begin{lemma}\label{lemma:1frac2}
The following assertions hold.
\begin{enumerate}
\item
$G$ has no fixed points on $X$;
\item
$G$ acts transitively on $\Sing'(X)$;
\item
every non-Gorenstein point $P\in X$ is cyclic quotient singularity of type $\frac12(1,1,1)$;
\item
for the stabilizer $G_P$ of $P\in \Sing'(X)$
there are the following possibilities:
\begin{enumerate}
\item
$G_P\simeq \Alt_6$, $N=7$ or $14$,
\item
$G_P\simeq \PSL_2(\bF_7)\times \mumu_3$, $N=15$.
\end{enumerate}
\end{enumerate}
\end{lemma}

\begin{proof}
Take a point $P\in \Sing'(X)$.
Let $r$ be the index of $P$, let $\Omega$ be its orbit, and let
and $n:=|\Omega|$.
Bogomolov--Miyaoka inequality \cite{Kawamata-1992bF}, \cite{KMMT-2000} gives us
\begin{equation*}
0< -K_X\cdot c_2=24 -\sum (r_i-1/r_i)=24 -3N/2\le 24 -3n/2.
\end{equation*}
Then using the list of maximal subgroups from Table~\ref{table:maximal}
one can obtain the following possibilities:
\begin{enumerate}
\item
$G_P\simeq 3.\Alt_6$, $n=7$,
\item
$G_P\simeq \PSL_2(\bF_7)\times \mumu_3$, $n=15$.
\end{enumerate}
Then one can proceed similarly to \cite[Lemma~6.1]{Prokhorov2012}.
\end{proof}

\begin{corollary}\label{cor:A7}
Let $\sigma: X_P\to X$ be the blowup of $P\in \Sing'(X)$ and let $E_P=\sigma^{-1}(P)$
be the exceptional divisor. Then $X_P$ is smooth along $E_P$, $E_P\simeq\PP^2$,
$\OOO_{E_P}(E_P)\simeq \OOO_{\PP^2}(-2)$, and the action of $G_P$ on $E_P$
has no fixed points.
\end{corollary}

\subsection{}
First we consider the case $|-K_X|=\emptyset$.

\begin{slemma}
In the above notation, $N=14$ or $15$,
\begin{eqnarray}
\label{eq:A7-K3}
(-K_X)^3 &=&\textstyle \frac12N -6\ge 1,
\\
\label{eq:A7-dim2K}
\dim |-2K_X|&=&\textstyle N -11\ge 3.
\end{eqnarray}
\end{slemma}

\begin{slemma}\label{lemma:A7-fco}
The linear system $|-2K_X|$ has no fixed components.
\end{slemma}
\begin{proof}
As in \cite[Claim~6.8.1]{Prokhorov2012} one can show that the Weil divisor class group $\Cl(X)$
is torsion free. Thus $\Cl(X)^G\simeq \ZZ$. Let $A$ be the ample generator of this group.
Write $-K_X=aA$. Since $K_X$ is not Cartier, $a$ is odd. On the other hand,
$a^3A^3=(-K_X)^3=1$ or $3/2$. Since $aA^3\in \frac12 \ZZ$,
this implies $a=1$, i.e.~$\Cl(X)^G\simeq \ZZ\cdot K_X$.
Since $|-K_X|=\emptyset$ the assertion follows.
\end{proof}

\begin{lemma}
\label{lemma:2K-lc}\label{lemma:A7-lc}
Let $\MMM=|-2K_X|$. Then the log pair $(X,\frac{3}{2}\MMM)$ is lc.
\end{lemma}

\begin{proof}
Suppose that $(X,\frac{3}{2}\MMM)$ is not lc.
We seek a contradiction.
Choose $\mu<\frac{3}{2}$ such that $(X,\mu\MMM)$ is strictly lc.
Let $Z$ be a minimal center of lc singularities of the log pair $(X,\mu\MMM)$.
Then $Z$ is either a point or a curve, because the base locus of $|-2K_X|$ does not have surfaces by
Lemma~\ref{lemma:A7-fco}.

Observe that $\theta (Z)$ is also a minimal center of lc singularities of this log pair for every $\theta \in G$.
Moreover, for every $\theta \in G$, either $Z\cap \theta (Z)=\emptyset$ or $Z=\theta (Z)$ by \cite[Proposition~1.5]{Kawamata1997}.
Using the perturbation trick
(see \cite[Lemma~2.4.10]{CheltsovShramov2015CRC} or the~proofs of \cite[Theorem~1.10]{Kawamata1997} and \cite[Theorem~1]{Kawamata-1998}),
for every sufficiently small $\epsilon>0$,
we can replace the boundary $\mu\MMM$ by an effective boundary $B_X$ such that
\begin{equation*}
B_X\qsim -2(\mu+\epsilon) K_X,
\end{equation*}
the log pair $(X,B_X)$ is strictly lc, and all its (not necessarily minimal) centers of lc singularities
are the subvarieties $\theta (Z)$ for $\theta \in G$.
In particular, there are minimal log canonical centers of the log pair $(X,B_X)$.
We may assume that $\mu+\epsilon<\frac{3}{2}$, since $\mu<\frac{3}{2}$.

Let $\Sigma$ be the union of all log canonical centers $\theta (Z)$ for $\theta \in G$.
Then $\Sigma$ is either a $G$-orbit or a disjoint union of irreducible isomorphic curves, which are transitively permuted by $G$.
In the latter case, each such curve is smooth by Kawamata's \cite[Theorem~1]{Kawamata-1998}.
Let $\III_{\Sigma}$ be the ideal sheaf of the locus $\Sigma$. Then
\begin{equation*}
h^1\left(\OOO_X\big(-2K_X\big)\otimes\III_{\Sigma}\right)=0,
\end{equation*}
by \cite[Theorem~9.4.17]{Lazarsfeld2004} or \cite[Theorem~2.16]{KollarPairs}, because $-2K_X-(K_X+B_X)$ is ample.
In particular, if $Z$ is a point, we see that
\begin{multline*}
|\Sigma|=h^0\left(\OOO_{\Sigma}\otimes\OOO_X\big(-2K_X\big)\right)=\\
=h^0\left(\OOO_X\big(-2K_X\big)\right)-h^0\left(\OOO_X\big(-2K_X\big)\otimes\III_{\Sigma}\right)\leqslant 5,
\end{multline*}
which must be a point since $3.\Alt_7$ does not have nontrivial subgroups of
index $\le 5$.
This is impossible by Lemma~\ref{lemma:fixed-point}.

Thus, we see that $\Sigma$ is a disjoint union of irreducible isomorphic smooth curves.
Denote them by $C_1=Z,C_2,\ldots,C_r$.
Let $d=-2K_X\cdot C_i$ and let $g$ be the genus of the curve $C_1$.
By Kawamata's subadjunction \cite[Theorem~1]{Kawamata-1998}, for every ample $\QQ$-divisor $A$ on $X$, we have
\begin{equation*}
\left(K_{X}+B_{X}+A\right)\Big\vert_{C_i}\qsim K_{C_i}+B_{C_i}
\end{equation*}
for some effective $\QQ$-divisor $B_{C_i}$ on the~curve $C_i$. This gives $d>2g-2$, so that $d\geqslant 2g-1$.
Thus, we have
\begin{multline*}
r(d-g+1)=h^0\left(\OOO_{\Sigma}\otimes\OOO_X\big(-2K_X\big)\right)=\\
=h^0\left(\OOO_X\big(-2K_X\big)\right)-h^0\left(\OOO_X\big(-2K_X\big)\otimes\III_{\Sigma}\right)\leqslant 5
\end{multline*}
by the Riemann--Roch formula applied to each curve $C_i$.
Now, using $r(d-g+1)\leqslant 5$ and $d\geqslant 2g-1$, we deduce that $r=1$ and $g\leqslant 5$.
This implies that $Z=C_1$ is pointwise fixed by $G$,
since $G$ cannot act non-trivially on a smooth curve of genus $\leqslant 5$.
This is a contradiction, since $G$ does not fix a point in $X$ by Lemma~\ref{lemma:fixed-point}.
\end{proof}

By \eqref{eq:A7-dim2K} we have that the dimension of $H^0(X,\OOO_X(-2K_X))$ is at most $5$.
Hence the action of $G=3.\Alt_7$ on this space is trivial.

Let $S$ be a general surface in $|-2K_X|$.
We claim that $S$ is normal.
Indeed, it follows from Lemma~\ref{lemma:A7-lc} that $(X,\MMM)$ is lc.
Then, by \cite[Theorem~4.8]{KollarPairs}, the log pair $(X,S)$ is also lc,
so that $S$ has lc singularities by \cite[Theorem~7.5]{KollarPairs}.
In particular, the surface $S$ is normal.
Take another general surface $S'\in |-2K_X|$ and consider the invariant curve $S\cap S'$. Write
\begin{equation*}
S\cap S'= \sum m_i C_i.
\end{equation*}
Put $d:=-2K_X\cdot C_i$. Since $-2K_X$ is an ample Cartier divisor, the numbers $d_i$ are integral and positive. Then
by \eqref{eq:A7-K3}
\begin{equation*}
\sum m_i d_i= (-2K_X)\cdot (S\cap S')= (-2K_X)^3=4(N -12)= 8 \quad \text{or $12$}.
\end{equation*}
From Table~\ref{table:maximal}, we see that $G=3.\Alt_7$ has at least one invariant component, say $C_1$.
We have
\begin{equation*}
-2K_S\cdot C_1= \left (\sum m_i C_i\right)\cdot C_1\le (-2K_X)\cdot \left (\sum m_i C_i\right) \le (-2K_X)^3\le 12.
\end{equation*}
In particular, $C_1^2\le 12$ and $K_S\cdot C_1\le 6$. Then by the genus formula
\begin{equation*}
2p_a(C_1)-2= (K_S+C_1)\cdot C_1 \le 18,\quad p_a(C_1)\le 10.
\end{equation*}
But the according to the Hurwitz bound the action of $G=3.\Alt_7$ on $C_1$ must be trivial.
This contradicts Lemma~\ref{lemma:fixed-point}.
Thus the case $|-K_X|=\emptyset$ does not occur.

\subsection{}
Consider the case $\dim |-K_X|=0$. Then $(-K_X)^3=\frac12 N-4\le \frac72$.
Let $S\in |-K_X|$ be the unique anticanonical member.
By Theorem~\ref{th:actions-K3} the singularities of $S$ are worse than Du Val.
Since $G$ has no fixed points, by Lemma~\ref{lemma:inv-hyp} the pair $(X,S)$ is lc and
$S$ is reducible: $S=\sum _{i=1}^m S_i$.
Then $\sum (-2K_X)^2\cdot S_i \le 14$.
So the cardinality of the orbit of $S_1$ equals $7$.
Let $\nu: S'\to S_1$ be the normalization.
Then by the adjunction $K_{S'}+D'=\nu ^*(K_X+S)|_{S_1}\sim0$ and the pair $(S',D')$ is lc.
Since $D'\neq 0$, the surface $S'$ is either rational or birationally
equivalent to a ruled surface over an elliptic curve. On the other hand,
The pair $(S', D')$ has a faithful action
of the stabilizer $3.\Alt_6\subset 3.\Alt_7$.
This is impossible.

\subsection{}
Now we consider the case $\dim |-K_X|>0$.

\begin{slemma}[{\cite[Lemma~6.6]{Prokhorov2012}}]
The pair $(X, |-K_X |)$ is canonical and therefore
a general member $S\in |-K_X|$ is a K3 surface with Du Val singularities.
\end{slemma}

Let $\sigma: Y\to X$ be the blowup of all non-Gorenstein points
and let $E=\sum E_i$ be the exceptional divisor.
Thus $Y$ has at worst terminal Gorenstein singularities
and it is smooth near $E$.
Since $(X, |-K_X |)$ is canonical, the linear system $|-K_Y|$ is the birational transform of
$|-K_X|$.
Put
\begin{equation*}
g:=\dim |-K_X|-1 =\textstyle \frac12(-K_X)^3 -\frac14N + 1,
\end{equation*}

\begin{slemma}[{\cite[Lemma~6.7]{Prokhorov2012}}]
\label{lemma:big}
The image of the ($G$-equivariant) rational map $\Phi: X\to \PP^{g+1}$
given by the linear system $|-K_X|$ is three-dimensional.
\end{slemma}
\begin{proof}
Suppose that $\dim \Phi(X)<3$.
Since $X$ is rationally connected, $G$ acts trivially on
$\Phi(X)$ and on $\PP^{g+1}$, this contradicts Theorem~\ref{th:actions-K3}.
\end{proof}

\begin{slemma}
The divisor $-K_{Y}$ is nef and big.
\end{slemma}

\begin{proof}
Assume that $-K_Y$ is not nef. Then $-K_Y\cdot C'<0$ for some
curve $C'$. Let $C$ be the $G$-orbit of $C'$. Then $-K_Y\cdot C<0$
and $C$ is $G$-invariant.
Note that $C\cap E_i$ is contained in
the base locus of the restricted linear system
$|-K_Y||_{E_i}$ which is a linear system of lines.
Thus $C\cap E_i$ is a $G_{E_i}$-invariant point on $E_i$.
This contradicts Corollary~\ref{cor:A7}.
Thus $-K_Y$ is nef. By Lemma~\ref{lemma:big} it is big.
\end{proof}

\begin{slemma}
The linear system $|-K_Y|$ is base point free and defines a crepant
birational morphism
\begin{equation*}
\Phi: Y \longrightarrow \bar Y\subset \PP^{g+1}
\end{equation*}
whose image $\bar Y$ is a Fano threefold with at worst canonical Gorenstein
singularities. Moreover, $\bar Y$ is an intersection of quadrics.
\end{slemma}
\begin{proof}
Follows from Lemmas~\ref{lemma:base-points}, \ref{lemma:very-ample}, and~\ref {lemma:trigonal}.
\end{proof}

Let $\Pi_i:=\Phi(E_i)$. Then $\Pi_1,\dots, \Pi_N$ are planes in $\PP^{g+1}$. Fix a plane, say $\Pi_1$
and let $G_1\subset G$ be its stabilizer.
Suppose that $\Pi_1\cap \Pi_i:=l$ is a line for some $i$.
Then the $G_1$-orbit of $l$ is given on $\Pi_1\simeq \PP^2$ by an invariant
polynomial, say $\phi$, which is a product of linear terms.
By \cite[P.~412]{Cohen-1976} one can see that $\deg \phi\ge 45$
if $G_1\simeq 3.\Alt_6$ and $\deg \phi\ge 21$
if $G_1\simeq \PSL_2(\bF_7)\times \mumu_3$. But this implies that $N\ge 21$,
a contradiction.

If $\Pi_1\cap \Pi_i:=p$ is a point for some $i$, then we can argue as above because
by duality the $G_1$-orbit of $p$ has at least 21 elements.

Therefore, the planes $\Pi_1,\dots, \Pi_N$ are disjoint.
Then $\Phi$ is an isomorphism and $Y$ is a Fano threefold with terminal
Gorenstein singularities and $\rk \Pic(Y)\ge 8$
because the divisors $E_i$ are linear independent elements of $\Pic(Y)$.
Moreover, $Y$ is $G\QQ$-factorial and $\rk \Pic(Y)^G=2$.
There exists an $G$-extremal Mori contraction $\varphi: Y\to Z$
which is different from $\sigma: Y\to X$.
Now $\varphi$ is birational and not small.
But then the $\varphi$-exceptional divisor $D$ meets $E$ and so
none of the components of $D$ are contracted to points.
Therefore, $Z$ is a Fano threefolds with $G\QQ$-factorial terminal Gorenstein singularities and
$\rk \Pic(Y)^G=1$.
This contradicts the above considered case~\ref{subsection:Gorenstein}.

\section{Amitsur Subgroup}

Here we review linearizations of line bundles and define a useful
equivariant birational invariant. Much of this simply mirrors known
results in the arithmetic setting, but our proofs have a more geometric
flavor. First, we review some facts about linearization of line bundles;
see \S{1}~and~\S{2}~of~\cite{Dolgachev97} for a more thorough discussion.

Let $X$ be a proper complex variety with a faithful action of a finite
group $G$.
One defines a morphism of $G$-linearized line bundles to be a morphism
of line bundles such that the map on the total spaces is equivariant.
We denote the group of isomorphism classes of $G$-linearized line
bundles by $\Pic(X,G)$.
Note that a line bundle $\LLL$ is $G$-invariant if and only if
$[\LLL] \in \Pic(X)^G$.
There is an evident group homomorphism
$\Pic(X,G) \to \Pic(X)^G$ obtained by forgetting the linearization.

Given a $G$-invariant line bundle $\LLL$, one constructs a cohomology
class $\delta(\LLL) \in H^2(G,\CC^\times)$ as follows.
Select an arbitrary isomorphism
$\phi_g : g^\ast \LLL \to \LLL$ for each $g \in G$.
Recall that any automorphism of a line bundle corresponds to
muliplication by a non-zero scalar since $X$ is proper.
Define a function $c : G \times G \to \CC^\times$
via
\begin{equation*}
c(g,h) := \phi(gh) \left( \phi_h \circ h^\ast( \phi_g) \right)^{-1}
\end{equation*}
for all $g,h \in G$.
One checks that $c$ is a $2$-cocycle and its cohomology class is independent
of the isomorphism class of the line bundle.

We have the following exact sequence of abelian groups:
\begin{equation*}
1 \to \Hom(G,\CC^\times) \to \Pic(X,G) \to \Pic(X)^G \to
H^2(G,\CC^\times) \ .
\end{equation*}
We define the \emph{Amitsur subgroup} as the group
\begin{equation*}
\Am(X,G)
:= \operatorname{im}( \Pic(X)^G \to H^2(G,\CC^\times) ) \ .
\end{equation*}
This is the name used for the arithmetic version in \cite{Liedtke}.

Note that $\Am(X,G)$ is a contravariant functor in $X$ via pullback of
line bundles. In fact, it is actually a birational invariant of smooth
projective $G$-varieties. This is well known in the arithmetic case
(see, for example, \S{5}~of~\cite{CTKM}).

\begin{theorem} \label{thm:Aminvariant}
If $X$ and $Y$ are smooth projective $G$-varieties that are
$G$-equivariantly birationally equivalent, then
$\Am(X,G) = \Am(Y,G)$.
\end{theorem}

\begin{proof}
First, assume that the theorem holds in the case where
$X \to Y$ is a blow-up of a smooth $G$-invariant subvariety.

Let $f : X \dasharrow Y$ be a $G$-equivariant birational map.
By \cite{ReichsteinYoussin}, we may resolve indeterminacies
and obtain a sequence of equivariant blow-ups $Z \to X$ with smooth
$G$-invariant centers
and an equivariant birational morphism $Z \to Y$.
By functoriality, we have an inclusion $\Am(Y,G) \subseteq \Am(Z,G)$.
By assumption, we have equality $\Am(Z,G)=\Am(X,G)$.
Thus $\Am(Y,G)$ is naturally a subset of $\Am(X,G)$.
Repeating the same argument for $f^{-1}$ shows that the two sets are
equal.

We now may assume that $\pi : X \to Y$ is a blow up of a
smooth $G$-invariant subvariety $C$.
Let $E$ be the exceptional divisor on $X$.
We have $\Pic(X)^G = \pi^\ast\Pic(Y)^G \oplus \ZZ [\OOO_X(E)]$
and so $\Am(X,G) = \Am(Y,G) + \ZZ\delta([\OOO_X(E)])$.
To complete the proof, we show that $\OOO_X(E)$ is $G$-linearizable.

Let $\LLL$ be a very ample line bundle on $X$ giving an embedding
$X \subseteq \PP^n$ for some $n$.
Since $H^2(G,\CC^\times)$ is torsion,
by replacing $\LLL$ by a sufficiently divisible power $\LLL^{\otimes n}$
we may assume that $\LLL$ is $G$-linearizable and that
$C = X \cap L$ where $L$ is a linear subspace of $\PP^n$.
We obtain $E$ as the pullback of the exceptional divisor of the blow-up
of $\PP^n$ along $L$.

Thus, the theorem reduces to showing that the exceptional divisor $E$
on the blow-up $X$ of a linear subspace $L$ of $\PP^n$ with dimension $m$, is
$G$-linearizable where $G$ has a linearizable action on $\PP^n$.
This follows from Lemma~\ref{lem:blowupLinearizable} below.
\end{proof}

\begin{lemma} \label{lem:blowupLinearizable}
Let $X$ be the blow-up of $\PP^n$ along a linear subspace $L$ of
dimension $m$.
Let $G$ be a finite group acting faithfully on $X$.
Then the line bundle associated to the exceptional divisor is
$G$-linearizable.
\end{lemma}

\begin{proof}
In this case, $X$ is a toric variety whose Cox ring $R$ may be described as
the polynomial ring
\begin{equation*}
R := \CC[x_0,\ldots,x_m,y_1,\ldots,y_{n-m},z]
\end{equation*}
where the monomials are graded by the Picard group.
If $H$ is the pullback of the hyperplane section of $X$ and
and $E$ is the exceptional divisor, then
$\deg(x_i)=[H]$, $\deg(y_i)=[H-E]$ and $\deg(z)=[E]$.
Let $S$ be the Neron-Severi torus dual to $\Pic(X)$.
For a multihomogeneous element $m \in R$ with grading $aH+bE$,
then $(\lambda,\mu) \in S$ acts on $S$ via
$(\lambda,\mu)\cdot (m) = \lambda^a\mu^b$.
The variety $X$ is obtained as the quotient of $\Spec(R)\setminus C$ by
$S$ where $C$ is the union of the subspaces defined by
$x_0=\cdots=x_m=0$ and $y_1=\cdots=y_{n-m}=z=0$.

From \S{4}~of~\cite{Cox},
the group of invertible elements among the graded ring
of endomorphisms of $R$ form a group $\widetilde{\Aut}(X)$
normalizing $S$ with quotient $\Aut(X)$.
The group $\widetilde{\Aut}(X)$ is isomorphic to
$U \rtimes (\GL_{m+1}(\CC) \times \GL_{n-m}(\CC) \times \CC^\times)$
where $\CC^\times$ acts by scalar multiplication on $z$,
$\GL_{m+1}$ acts linearly on $x_0,\ldots, x_m$,
$\GL_{n-m}$ acts linearly on $y_1,\ldots, y_{n-m}$
and elements $u \in U$ are all of the form $\operatorname{id}+n$
where $n$ is a linear map from
the span of $x_0, \ldots, x_m$ to $zy_1, \ldots, zy_{n-m}$.
Since $G$ is finite, we may assume that $G$ has a preimage in
$\tilde{G}$ in the subgroup $\SL_{m+1}(\CC) \times \SL_{n-m}(\CC)$
of $\widetilde{\Aut}(X)$.
Thus $\tilde{G} \cap S = (\mu_d,1)$ where $d=\gcd(m+1,n-m)$.

Recall that the canonical bundle on $\PP^n$ is always linearizable,
so to show that $E$ is linearizable it suffices to show that
that $G$ has an action on the global sections of the
very ample line bundle $\OOO_X(E+d(n+1)H)$.
From \cite{Cox}, the vector space $H^0(X,\OOO_X(E+d(n+1)H))$
is isomorphic to vector subspace of $R$ with
grading $[E+d(n+1)H]$.
This is spanned by monomials $x_i^ay_j^bz^c$
where $a+b=d(n+1)$ and $c=b+1$.
Now $s=(\lambda,\mu) \in S$ acts via $s(m)=\lambda^{d(n+1)}\mu m$ on every such
monomial, thus $\tilde{G} \cap S$ acts trivially.
We conclude the action $\tilde{G}$ factors through $G$ as desired.
\end{proof}

\begin{remark}
Note that there is much shorter proof of
Theorem~\ref{thm:Aminvariant} when $X$ and $Y$ are surfaces
(which is actually the only case we need).
One reduces to the case where $\pi : X \to Y$ is a blow-up of
a $G$-orbit of points. We only need to show that the exceptional
divisor $E$ is $G$-linearizable as before.
This is immediate since $K_X = \pi^\ast K_Y + E$ and
both $K_X$ and $K_Y$ are $G$-linearizable.
This proof fails in higher dimensions since the exceptional divisor
may appear with multiplicity.
\end{remark}

We have the following due to Proposition~2.2~of~\cite{Dolgachev97}:

\begin{proposition} \label{prop:AMcurve}
If $X$ is a smooth curve with a faithful action of $G$,
then $\Am(X,G)=H^2(G,\CC^\times)$ for all subgroups $G \subseteq \Aut(X)$.
\end{proposition}

Since the canonical bundle is always linearizable, we have the following:

\begin{proposition} \label{prop:AMFanoIndex1}
If $X$ is a smooth Fano variety of index $1$ with $\rk \Pic(X)^G =1$,
then $\Am(X,G)$ is trivial for all subgroups $G \subseteq \Aut(X)$.
\end{proposition}

\begin{lemma} \label{lem:AMsection}
If $f : X \to Y$ is a $G$-equivariant morphism,
then the morphism $f^\ast : \Am(Y,G) \to \Am(X,G)$ is injective.
If moreover, $f$ has a $G$-equivariant section, then
$\Am(Y,G)$ is a direct summand of $\Am(X,G)$.
\end{lemma}

\begin{proof}
Let $\LLL$ be an element of $\Pic(Y)^G$ and suppose $E$ is the extension
of $G$ that acts on the total space of $\LLL$.
Then $E$ also acts on $f^\ast \LLL$. If $f^\ast \LLL$ is
$G$-linearizable, then $E$ splits; thus $\LLL$ must also be
linearizable.
Now suppose $s : Y \to X$ is a section.
By definition $f \circ s = \operatorname{id}_Y$, so the map induced by
functoriality $\Am(Y,G) \to \Am(X,G) \to \Am(Y,G)$ is the identity.
\end{proof}

It is easy to determine the possible values of the invariant
for rational surfaces
(c.f. Proposition~5.3~of~\cite{CTKM} for the arithmetic case).

\begin{proposition}
Suppose $X$ is a rational surface with $G \subseteq \Aut(X)$.
We have the following possibilities:
\begin{enumerate}
\item
$X$ is $G$-equivariantly birationally equivalent to $\PP^2$
and $\Am(X,G)$ is isomorphic to $\ZZ/3\ZZ$ or $0$.
\item
$X$ is $G$-equivariantly birationally equivalent to $\PP^1 \times
\PP^1$,
and $\Am(X,G)$ is isomorphic to
$\ZZ/2\ZZ \times \ZZ/2\ZZ$, $\ZZ/2\ZZ$ or $0$.
\item
$X$ is $G$-equivariantly birationally equivalent to $\bF_n$.
If $n$ is odd then $\Am(X,G)=0$. If $n \ge 2$ is even,
then $\Am(X,G)$ is isomorphic to $\ZZ/2\ZZ$ or $0$.
\item
If $X$ is $G$-equivariantly birationally equivalent to a
minimal conic bundle surface with singular fibers,
then $\Am(X,G)$ is isomorphic to $\ZZ/2\ZZ$ or $0$.
\item
Otherwise, $X$ is isomorphic to a del Pezzo surface of degree
$\le 6$ and $\Am(X,G)$ is trivial.
\end{enumerate}
\end{proposition}

\begin{proof}
It suffices to assume $X$ is a $G$-minimal surface.

If $X$ is a $G$-minimal del Pezzo surface, then either $X=\PP^2$,
$X=\PP^1 \times \PP^1$, or it has Fano index $1$.
In the last case, $\Am(X,G)=0$ by Proposition~\ref{prop:AMFanoIndex1}.

For $\PP^2$, the index is $3$ and there exist non-linearizable line
bundles. Thus $\Am(X,G)=0$ or $\ZZ/3\ZZ$.

For $\PP^1 \times \PP^1$, we consider two cases.
First, assume the fibers are interchanged by $G$.
Then $\Pic(X)^G \simeq \ZZ$ and $-K_X$ has index $2$.
Not every group is linearizable, so we have $\Am(X,G)=0$ or
$\ZZ/2\ZZ$.

Now suppose the fibers are not interchanged.
Taking $G=G_1 \times G_2$ with each $G_i$ acting on each $\PP^1$
separately, one may have $\OOO(1,0)$ and $\OOO(0,1)$ be non-linearizable
with distinct classes in $H^2(G,\CC^\times)$.
Thus $\ZZ/2\ZZ \times \ZZ/2\ZZ$ and all its subgroups are a possibility
for $\Am(X,G)$.

Suppose $X$ is a ruled surface $\bF_n$.
Since we have a $G$-invariant section, by Lemma~\ref{lem:AMsection}
$\Am(X,G)=\Am(\PP^1,G)$ so it only depends on the group $G$.
From \cite{DolgachevIskovskikh}, we see that the reductive part of $\Aut(X)$
is isomorphic to $\CC^\times \rtimes \PSL_2(\CC)$ if $n$ is even and
$\CC^\times \rtimes \SL_2(\CC)$ if $n$ is odd.
The subgroup $\CC^\times$ acts trivially on the base,
so $\Am(X,G)=0$ if $n$ is odd, but can be $\Am(X,G)=\ZZ/2\ZZ$ if $n$ is
even.

If $X \to \PP^1$ is a minimal $G$-conic bundle,
then $\Pic(X)^G$ is generated by $-K_X$ and $\pi^\ast \Pic(\PP^1,G)$.
Since $-K_X$ is always linearizable, $\Am(X)$ is either trivial
or $\ZZ/2\ZZ$.
Note that, for example, $\ZZ/2\ZZ \times \Alt_5$ acts on an exceptional
conic bundle (see Proposition~5.3~of~\cite{DolgachevIskovskikh}).
Thus $\Am(X)$ can be non-trivial.
\end{proof}

Note that the Amitsur subgroup can distinguish between equivariant
birational equivalence classes quickly that might be more involved
using other methods.

\begin{example}
The action of the group $\Alt_5$ on $\PP^2$ has trivial Amitsur subgroup
$\Am(\PP^2,\Alt_5)$ since $\Alt_5 \subseteq \PGL_3(\CC)$ lifts to
$\GL_3(\CC)$.
The action of $\Alt_5$ on $\PP^1 \times \PP^1$ is not linearizable
since $\Alt_5 \subseteq \PGL_2(\CC)$ does not lift to $\GL_2(\CC)$.
Thus $\Am(\PP^2,\Alt_5)$ and $\Am(\PP^1 \times \PP^1,\Alt_5)$ are not
equal.
Thus $\PP^2$ and $\PP^1 \times \PP^1$ are not $\Alt_5$-birationally
equivalent.
This also follows from \cite[Theorem~6.6.1]{CheltsovShramov2015CRC}.
\end{example}

\section{Conic bundles}
\label{section:conic-bundles}

Let $X$ be a smooth projective variety, and let $G$ be a finite subgroup in $\Aut(X)$.
Suppose that there exists a $G$-equivariant conic bundle $\eta\colon X\to Y$ such that $Y$ is smooth, and the morphism $\eta$ is flat.
These assumptions mean that $\eta$ is regular conic bundle in the sense of \cite[Definition~1.4]{Sarkisov1982}.
Note that $G$ naturally acts on $Y$. But this action is not necessarily faithful.
In general, we have an exact sequence of groups
\begin{equation*}
1\longrightarrow G_{\eta}\longrightarrow G\longrightarrow G_Y\longrightarrow 1,
\end{equation*}
where $G_Y$ is a subgroup in $\Aut(Y)$, and the subgroup $G_\eta$ acts trivially on $Y$.
If
\begin{equation*}
\Pic(X)^G=\eta^*\Pic(Y)^{G_Y}\oplus\ZZ,
\end{equation*}
we say that the conic bundle $\eta\colon X\to Y$ is a $G$-minimal or $G$-standard (cf. \cite[\S1]{Prokhorov2018} and \cite[Definition~1.12]{Sarkisov1982}).
In this case, the conic bundle $\eta\colon X\to Y$ is a $G$-Mori fiber space.

\begin{lemma}
\label{lemma:conic-bundle-zenter-action-smooth-fibers}
Let $G_1\subset G_\eta$ be a non-trivial subgroup and let $\Fix(G_1)$ be its fixed point set.
Then $\Fix(G_1)$ does not contain any component of a reduced fiber.
\end{lemma}

\begin{proof}
Let $C=\eta^{-1}(Q)$ be a reduced fiber and let $C_1\subset C$ be an irreducible component
and let $P\in C$ be a smooth point
Assume that $\Fix(G_1)\supset C_1$.
Consider the exact sequence
\begin{equation}\label{equation:exact-sequence-C1-P}
0\longrightarrow T_{P,C}\longrightarrow T_{P,X}\longrightarrow \eta^*\big(T_{Q,Y}\big)\longrightarrow 0.
\end{equation}
If $G_1$ does not act faithfully on the curve $C_1$,
then it does not act faithfully on the tangent space $T_{P,C}$,
and it also acts trivially on the tangent space $T_{P,Y}$.
Thus, in this case, $G_1$ does not act faithfully on $T_{P,X}$, which is impossible by Lemma~\ref{lemma:stabilizer-faithful},
since $G_1$ acts faithfully on $X$ by assumption.
\end{proof}

Suppose, in addition, that $G_{\eta}\simeq\mumu_n$ for $n\geqslant 2$.
Let $B$ be the union of codimension one subvarieties in $X$ that are pointwise fixed by the subgroup $G_{\eta}$.
Then $B$ is smooth (see Lemma~\ref{lemma:stabilizer-faithful}).
Since all smooth fibers of the conic bundle $\eta$ are isomorphic to $\PP^1$,
we see that the subgroup $G_{\eta}$ fixes exactly two points in each (see Lemma~\ref{lemma:conic-bundle-zenter-action-smooth-fibers}).
Thus, if $C$ is a general fiber of $\eta$, then the intersection $B\cap C$ consists of two distinct points.
This implies that $B\cdot C=2$ for every fiber $C$ of the conic bundle $\eta$.
Hence, the morphism $\eta$ induces a generically two-to-one morphism $\phi\colon B\to Y$.

\begin{lemma}
\label{lemma:conic-bundle-zenter-action}
Let $C$ be a reduced fiber of the conic bundle $\eta$.
If $n=|G_{\eta}|\geqslant 3$, then $G_{\eta}$ leaves invariant every irreducible component of $C$.
\end{lemma}

\begin{proof}
We may assume that $C$ is reducible.
Then $C=C_1\cup C_2$, where $C_1$ and $C_2$ are smooth irreducible rational curves that intersects transversally at one point.
Denote this point by $O$.
Then $O$ is fixed by $G_{\eta}$.

If $n=|G_{\eta}|$ is odd, then both $C_1$ and $C_2$ are $G_{\eta}$-invariant.
Thus, to complete the proof, we may assume that $n$ is even and neither $C_1$ nor $C_2$ is $G_{\eta}$-invariant.
Let us seek for a contradiction.

Let $z$ be a generator of the group $G_{\eta}$.
Then $z$ swaps the curves $C_1$ and $C_2$.
In particular, this shows that
\begin{equation*}
B\cap C=B\cap C_1=B\cap C_2=C_1\cap C_2=O.
\end{equation*}
On the other hand, the element $z^2$ leaves both curves $C_1$ and $C_2$ invariant.
By Lemma~\ref{lemma:conic-bundle-zenter-action-smooth-fibers},
we see that $z^2$ acts faithfully on both these curves.
Since $C_1\simeq\PP^1$, the element $z^2$ fixes a point $P\in C_1$ such that $P\ne O$.
Moreover, using \eqref{equation:exact-sequence-C1-P}, we see that there exists a two-dimensional subspace
in the tangent space $T_{P,X}$ consisting of zero eigenvalues of the element $z^2$.
By \cite[Theorem~2.1]{Birula}
this shows that $z^2$ pointwise fixes a surface $B_1$ in $X$ such that $P\in B_1$.

Since $P\not\in B$, we see that $B_1$ is not an irreducible component of the surface $B$.
On the other hand, the surface $B$ is also pointwise fixed by $z^2$, because it is pointwise fixed by $z$.
Let $C^\prime$ be a general fiber of the conic bundle $\eta$.
Then $G_{\eta}$ acts faithfully on $C^\prime$ by Lemma~\ref{lemma:conic-bundle-zenter-action-smooth-fibers},
so that $z^2$ fixes exactly two points in $C^\prime$.
On the other hand, it also fixes all points of the intersection $B_1\cap C^\prime$ and all points of the intersection $B\cap C^\prime$.
This shows that $z^2$ fixes at least three points in $C^\prime$, which is absurd.
The obtained contradiction shows that $z$ does not swap the curves $C_1$ and $C_2$,
which completes the proof of the lemma.
\end{proof}

Let $\Delta\subset Y$ be the discriminant locus of $\eta$, i.e.~the locus consisting of points $P\in Y$ such that
the scheme fiber $\eta^{-1}(P)$ is not isomorphic to $\PP^1$.
Then $\Delta$ is a (possibly reducible) reduced $G_Y$-invariant divisor that has at most normal crossing singularities in codimension $2$.
If $P$ is a smooth point of $\Delta$, then the fiber of $\eta$ over $P$ is isomorphic to a reducible reduced conic in $\PP^2$.
If $P$ is a singular point of $\Delta$, then $F$ is isomorphic to a non-reduced conic in $\PP^2$.

\begin{lemma}
\label{lemma:conic-bundle-B-Y-unramified}
Suppose that $G_{\eta}$ leaves invariant every irreducible component of each reduced fiber of the conic bundle $\eta$.
Then $\phi\colon B\to Y$ is an \'etale double cover over $Y\setminus\Sing(\Delta)$.
\end{lemma}

\begin{proof}
Fix a point $Q\in Y$ such that $Q\not\in \Sing(\Delta)$.
Let $C$ be the fiber of the conic bundle of $\eta$ over the point $Q$.
By Lemma~\ref{lemma:conic-bundle-zenter-action}, the center $G_{\eta}$ acts faithfully on every irreducible component of the fiber $C$.
In particular, we see that $C\not\subset B$.
Since $B\cdot C=2$, we see that either $|B\cap C|$ consists of two distinct points or $|B\cap C|$ consists of a single point.
In the former case, the morphism $\phi$ is \'etale over $Q$.
Thus, we may assume that the intersection $B\cap C$ consists of a single point.
Denote this point by~$P$.

Suppose first that $C$ is smooth. Then $C$ is tangent to $B$ at the point $P$.
Then $T_{P,C}\subset T_{P,B}$,
so that $G_{\eta}$ acts trivially on $T_{P,C}$.
This is impossible by Lemma~\ref{lemma:stabilizer-faithful}, since $G_{\eta}$ acts faithfully on $C$.

We see that $C=C_1\cup C_2$, where $C_1$ and $C_2$ are smooth irreducible rational curves that intersects transversally at one point.
Denote this point by $O$.
If $P\ne O$, then $T_{P,C_1}\subset T_{P,B}$.
As above, this leads to a contradiction, since $G_{\eta}$ acts faithfully on the curve $C_1$.
Thus, we have $O=P$.

Recall that $G_{\eta}$ is cyclic by assumption.
Denote by $z$ its generator.
Then $z$ must fix a point $P^\prime\in C_1$ that is different from $P$.
Using the exact sequence
\begin{equation*}
0\longrightarrow T_{P^\prime,C_1}\longrightarrow T_{P^\prime,X}\longrightarrow \eta^*\big(T_{Q,Y}\big)\longrightarrow 0,
\end{equation*}
we see that there exists a two-dimensional subspace in the tangent space $T_{P',X}$ consisting of zero eigenvalues of the element $z$.
By \cite[Theorem~2.1]{Birula} this shows that $z$ pointwise fixes a surface $B^\prime$ in $X$ such that $P^\prime\in B^\prime$.

Since $P=B\cap C$, we see that $P^\prime\not\in B$, so that $B^\prime$ is not an irreducible component of the surface $B$.
On the other hand, the surface $B^\prime$ is also pointwise fixed by $z$.
This contradicts the definition of the surface $B$.
\end{proof}

\begin{corollary}
\label{corollary:conic-bundle-B-Y-unramified}
If $|G_{\eta}|\geqslant 3$, then $\phi\colon B\to Y$ is an \'etale double cover over $Y\setminus\Sing(\Delta)$.
\end{corollary}

\begin{corollary}
\label{corollary:conic-bundle-main}
Suppose that $|G_{\eta}|\geqslant 3$ or $G_{\eta}$ leaves invariant every irreducible component of each reduced fiber of the conic bundle $\eta$.
Suppose also that $\eta\colon X\to Y$ is $G$-minimal, and $Y$ is simply connected.
Then $\Delta=0$, and there exists a central extension $\widetilde{G}_Y$ of the group $G_Y$ such that
\begin{equation*}
X\simeq\PP\big(\LLL_1\oplus\LLL_2\big)
\end{equation*}
for some $\widetilde{G}_Y$-linearizable line bundles $\LLL_1$ and $\LLL_2$ on $Y$,
where the splitting $\LLL_1\oplus\LLL_2$ is also $\widetilde{G}_Y$-invariant,
and $\eta\colon X\to Y$ is a natural projection.
\end{corollary}

\begin{proof}
By Lemma~\ref{lemma:conic-bundle-B-Y-unramified} and Corollary~\ref{corollary:conic-bundle-B-Y-unramified},
the morphism $\phi\colon B\to Y$ is an \'etale double cover over $Y\setminus\Sing(\Delta)$.
Since $Y$ is simply connected, we see that
\begin{equation*}
B=B_1\cup B_2,
\end{equation*}
where $B_1$ and $B_2$ are rational sections of the conic bundle $\eta$.
Now $G$-minimality implies that $\Delta=0$, so that $\eta$ is a $\PP^1$-bundle.
Since $\phi\colon B\to Y$ is an \'etale double cover, we see that
$B$ is a disjoint union of the divisors $B_1$ and $B_2$,
and both $B_1$ and $B_2$ are sections of the $\PP^1$-bundle $\eta$.
This  implies the remaining assertions of the corollary.
\end{proof}

Now we are ready to prove

\begin{theorem}
\label{theorem:conic-bundle-6-A-6}
Let $X$ be a rationally connected threefold that is faithfully acted by the group $6.\Alt_6$.
Then there exists no $6.\Alt_6$-equivariant dominant rational map $\pi\colon X\dasharrow S$ such that
its general fibers are irreducible rational curves irreducible rational surfaces.
\end{theorem}

\begin{proof}
Using $6.\Alt_6$-equivariant resolution of singularities and indeterminacies, we may assume that both $X$ and $S$ are smooth,
and $\pi$ is a morphism.
If $S$ is a curve, then we immediately obtain a contradiction, since $\Alt_6$ cannot faithfully act on a rational curve,
and $6.\Alt_6$ cannot faithfully act on a rational surface.
Thus, we may assume that $S$ is a surface, and general fiber of $\pi$ is $\PP^1$.
Using \cite[Theorem~1]{Avilov}, we may assume that $\pi\colon X\to S$ is $6.\Alt_6$-minimal standard conic bundle.

Since $X$ is rationally connected, we see that $S$ is rational.
Thus, none of the groups $6.\Alt_6$, $3.\Alt_6$ and $2.\Alt_6$ can faithfully act on $S$.
This shows that there exists an exact sequence of groups
\begin{equation*}
1\longrightarrow G_{\pi}\longrightarrow 6.\Alt_6\longrightarrow G_S\longrightarrow 1,
\end{equation*}
where $G_S$ is a subgroup in $\Aut(S)$ that is isomorphic to $\Alt_6$,
and $G_\pi$ is the center of the group $6.\Alt_6$ that acts trivially on $S$.
Applying Corollary~\ref{corollary:conic-bundle-main}, we see that
there exists a central extension $\widetilde{G}_S$ of the group $G_S\simeq\Alt_6$ such that
\begin{equation*}
X\simeq\PP\big(\LLL_1\oplus\LLL_2\big)
\end{equation*}
for some $\widetilde{G}_S$-linearizable line bundles $\LLL_1$ and $\LLL_2$ on $S$,
where the splitting $\LLL_1\oplus\LLL_2$ is also $\widetilde{G}_S$-invariant.

A priori, we know that $\widetilde{G}_S$ is one of the following groups $\Alt_6$, $2.\Alt_6$, $3.\Alt_6$ or $6.\Alt_6$.
On the other hand, the induced action of $\widetilde{G}_S$ on $X$ gives our action of $6.\Alt_6$ on the threefold $X$.
This shows that $\Am(S,\Alt_6) \simeq \mumu_6$.
On the other hand, it follows from \cite{DolgachevIskovskikh} that $S$ is $\Alt_6$-birational to $\PP^2$,
where $\Am(\PP^2,\Alt_6) \simeq \mumu_3$, contradicting Lemma~\ref{thm:Aminvariant}.
\end{proof}

\newcommand{\etalchar}[1]{$^{#1}$}


\begin{table}
{\scriptsize
\begin{tabular}{|c|c|cccccccccc|}
\hline
Group & Dim & \multicolumn{10}{c|}{Invariants in Low Degrees} \\
 & & 1 & 2 & 3 & 4 & 5 & 6 & 7 & 8 & 9 & 10 \\
\hline
\hline
\multirow{4}{*}{$\PSL_2(\bF_7)$}
& 3 & & & & 1 & & 1 & & 1 & & 1\\
& 6 & & 1 & 2 & 3 & 4 & 8 & 10 & 15 & 22 & 30\\
& 7 & & 1 & 1 & 4 & 2 & 10 & 10 & 25 & 28 & 58\\
& 8 & & 1 & 2 & 3 & 5 & 15 & 19 & 44 & 72 & 120\\
\hline
\multirow{3}{*}{$\SL_2(\bF_7)$}
& 4 & & & & 1 & & 1 & & 3 & & 2\\
& 6 & & & & 1 & & 2 & & 10 & & 16\\
& 8 & & & & 2 & & 10 & & 44 & & 106\\
\hline
\multirow{5}{*}{$\PSL_2(\bF_{11})$}
& 5 & & & 1 & & 1 & 2 & 1 & 2 & 3 & 3\\
& 10 & & 1 & & 4 & 1 & 16 & 10 & 54 & 56 & 176\\
& 10 & & 1 & 2 & 4 & 8 & 16 & 28 & 54 & 98 & 176\\
& 11 & & 1 & 1 & 3 & 4 & 20 & 24 & 78 & 134 & 300\\
& 12 & & 1 & 1 & 4 & 8 & 25 & 49 & 124 & 258 & 558\\
\hline
\multirow{4}{*}{$\SL_2(\bF_{11})$}
& 6 & & & & 1 & & 1 & & 4 & & 4\\
& 10 & & & & 1 & & 6 & & 44 & & 124\\
& 10 & & & & 3 & & 6 & & 44 & & 134\\
& 12 & & & & 4 & & 15 & & 124 & & 516\\
\hline
\multirow{4}{*}{$\Alt_6$}
& 5 & & 1 & 1 & 2 & 2 & 4 & 3 & 6 & 6 & 9\\
& 8 & & 1 & 1 & 2 & 3 & 9 & 9 & 23 & 34 & 60\\
& 9 & & 1 & 2 & 4 & 7 & 14 & 23 & 46 & 80 & 140\\
& 10 & & 1 & & 7 & 2 & 25 & 20 & 94 & 108 & 308\\
\hline
\multirow{3}{*}{$2.\Alt_6$}
& 4 & & & & & & & & 2 & & \\
& 8 & & & & 1 & & 4 & & 23 & & 46\\
& 10 & & & & 4 & & 11 & & 80 & & 248\\
\hline
\multirow{4}{*}{$3.\Alt_6$}
& 3 & & & & & & 1 & & & & \\
& 6 & & & 2 & & & 7 & & & 16 & \\
& 9 & & & 2 & & & 14 & & & 80 & \\
& 15 & & & 2 & & & 126 & & & 2234 & \\
\hline
\multirow{2}{*}{$6.\Alt_6$}
& 6 & & & & & & 1 & & & & \\
& 12 & & & & & & 29 & & & & \\
\hline
\multirow{7}{*}{$\Alt_7$}
& 6 & & 1 & 1 & 2 & 2 & 4 & 4 & 6 & 7 & 10\\
& 10 & & & & 3 & & 6 & 2 & 20 & 11 & 54\\
& 14 & & 1 & 2 & 4 & 8 & 21 & 42 & 105 & 233 & 506\\
& 14 & & 1 & 2 & 5 & 9 & 22 & 46 & 109 & 237 & 518\\
& 15 & & 1 & 1 & 5 & 4 & 25 & 45 & 150 & 320 & 826\\
& 21 & & 1 & 2 & 8 & 24 & 110 & 362 & 1284 & 4023 & 12046\\
& 35 & & 1 & 4 & 37 & 225 & 1582 & 8864 & 47098 & 223591 & 985678\\
\hline
\multirow{5}{*}{$2.\Alt_7$}
& 4 & & & & & & & & 1 & & \\
& 14 & & & & 1 & & 8 & & 94 & & 438\\
& 20 & & & & 3 & & 64 & & 919 & & 7845\\
& 20 & & & & 7 & & 68 & & 929 & & 7905\\
& 36 & & & & 38 & & 1749 & & 57807 & & 1264859\\
\hline
\multirow{6}{*}{$3.\Alt_7$}
& 6 & & & 1 & & & 3 & & & 5 & \\
& 15 & & & 1 & & & 23 & & & 314 & \\
& 15 & & & 3 & & & 33 & & & 404 & \\
& 21 & & & 2 & & & 110 & & & 4023 & \\
& 21 & & & & & & 120 & & & 3806 & \\
& 24 & & & 2 & & & 208 & & & 11146 & \\
\hline
\multirow{3}{*}{$6.\Alt_7$}
& 6 & & & & & & & & & & \\
& 24 & & & & & & 177 & & & & \\
& 36 & & & & & & 1749 & & & & \\
\hline
\end{tabular}
}
\caption{Representations and invariants of low degrees.}
\label{table}
\vspace{1in}
\end{table}

\begin{table}
{\scriptsize
\begin{tabular}{|c|c|cccccccccc|}
\hline
Group & Dim & \multicolumn{10}{c|}{Invariants in Low Degrees} \\
 & & 1 & 2 & 3 & 4 & 5 & 6 & 7 & 8 & 9 & 10 \\
\hline
\hline
\multirow{14}{*}{$\PSp_4(\bF_3)$}
& 5 & & & & 1 & & 1 & & 1 & & 2\\
& 6 & & 1 & & 1 & 1 & 2 & 1 & 3 & 2 & 4\\
& 10 & & & & 1 & & 2 & & 5 & 2 & 8\\
& 15 & & 1 & 1 & 3 & 6 & 13 & 21 & 48 & 90 & 180\\
& 15 & & 1 & 1 & 3 & 2 & 9 & 9 & 30 & 44 & 115\\
& 20 & & 1 & 2 & 5 & 10 & 26 & 56 & 151 & 380 & 980\\
& 24 & & 1 & 2 & 6 & 12 & 41 & 117 & 409 & 1268 & 4006\\
& 30 & & 1 & & 9 & 7 & 108 & 267 & 1785 & 5816 & 26198\\
& 30 & & & 2 & 5 & 15 & 89 & 361 & 1560 & 6526 & 25024\\
& 40 & & & & 9 & 38 & 361 & 1987 & 12432 & 64242 & 318717\\
& 45 & & & & 14 & 65 & 655 & 4365 & 29347 & 170291 & 925070\\
& 60 & & 1 & 4 & 34 & 320 & 3316 & 30266 & 252784 & 1903375 & 13127051\\
& 64 & & 1 & 2 & 38 & 409 & 4706 & 46253 & 411176 & 3283749 & 23975553\\
& 81 & & 1 & 4 & 94 & 1258 & 18430 & 225424 & 2483426 & 24523546 & 220742112\\
\hline
\multirow{9}{*}{$\Sp_4(\bF_3)$}
& 4 & & & & & & & & & & \\
& 20 & & & & 1 & & 7 & & 103 & & 765\\
& 20 & & & & 2 & & 9 & & 106 & & 783\\
& 20 & & & & & & 6 & & 96 & & 755\\
& 36 & & & & 3 & & 184 & & 5631 & & 122657\\
& 60 & & & & 26 & & 3148 & & 252182 & & 13116046\\
& 60 & & & & 29 & & 3153 & & 252155 & & 13116416\\
& 64 & & & & 34 & & 4579 & & 411176 & & 23967693\\
& 80 & & & & 79 & & 16801 & & 2256083 & & 196172329\\
\hline
\end{tabular}
}
\caption{Representations and invariants of low degrees (continued).}
\label{table2}
\vspace{1in}
\end{table}

\begin{table}
{\scriptsize
\begin{tabular}{|lr|lr|lr|lr|lr|lr|}
\hline
\multicolumn{2}{|c|}{$\Alt_5$} &
\multicolumn{2}{c|}{$\Alt_6$} &
\multicolumn{2}{c|}{$\Alt_7$} &
\multicolumn{2}{c|}{$\PSL_2(\bF_7)$} &
\multicolumn{2}{c|}{$\PSL_2(\bF_{11})$} &
\multicolumn{2}{c|}{$\PSp_4(\bF_3)$} \\
\hline
$\Alt_4$ & 5 &
$\Alt_5$ & 6 &
$\Alt_6$ & 7 &
$\Sym_4$ & 7 &
$\Alt_5$ & 11 &
$\mumu_2^4 \rtimes \Alt_5$ & 27 \\
$D_{10}$ & 6 &
$\Alt_5$ & 6 &
$\PSL_2(\bF_7)$ & 15 &
$\Sym_4$ & 7 &
$\Alt_5$ & 11 &
$\Sym_6$ & 36 \\
$\Sym_3$ & 10 &
$\mumu_3^2 \rtimes \mumu_4$ & 10 &
$\PSL_2(\bF_7)$ & 15 &
$\mumu_7 \rtimes \mumu_3$ & 8 &
$\mumu_{11} \rtimes \mumu_5$ & 12 &
$\operatorname{SU}_3(\bF_2) \rtimes \mumu_3$ & 40 \\
 & &
$\Sym_4$ & 15 &
$\Sym_5$ & 21 &
 & &
$D_{12}$ & 55 &
$\mumu_3^3 \rtimes \Sym_4$ & 40 \\
 & &
$\Sym_4$ & 15 &
$(\Alt_4 \times \mumu_3) \rtimes \mumu_2$ & 35 &
 & &
 & &
$\mumu_2 \cdot (\Alt_4 \wr \mumu _2)$ & 45 \\
\hline
\end{tabular}}
\caption{Conjugacy classes of maximal subgroups with indices.}
\label{table:maximal}
\end{table}


\end{document}